\documentclass{article}

\usepackage{units,stmaryrd,color,relsize,amsmath,amssymb,amsthm,stackrel}

\usepackage{enumerate}
\usepackage{hyperref}

\usepackage[all]{xy}

\newtheorem{theorem}{Theorem}[section]

\newtheorem{definition}[theorem]{Definition}
\newtheorem{lemma}[theorem]{Lemma}
\newtheorem{corollary}[theorem]{Corollary}

\newtheorem{defi}[theorem]{Definition}
\newtheorem{notation}[theorem]{Notation}
\newtheorem{example}[theorem]{Example}
\newtheorem{remark}[theorem]{Remark}

\newcommand{\PreprintVersusPaper}[2]{#2}

\newcommand{\formerOmega}{o}

\newcommand{\glp}{{\ensuremath{\mathsf{GLP}}}\xspace}
\newcommand{\RC}{{\ensuremath{\mathsf{RC}}}\xspace}

\usepackage{xspace}

\newcommand{\pa}{\ensuremath{{\mathrm{PA}}}\xspace}
\newcommand{\zfc}{\ensuremath{{\mathrm{ZFC}}}\xspace}

\newcommand{\gl}{{\ensuremath{\textup{{\sf GL}}}}\xspace}

\newcommand{\ea}{\ensuremath{{\rm{EA}}}\xspace}

\newcommand{\la}{\langle}
\newcommand{\ra}{\rangle}

\newcommand{\ord}{\mathsf{On}}

\newcommand\oOld{{\check o}}
\newcommand\oAlg{{\rm o}}
\newcommand\body[1]{b({#1})}

\newcommand\Worms{{\mathbb W}}

\def\x{{\xi}}
\def\y{\zeta}

\def\lb{\left\llbracket}
\def\rb{\right\rrbracket}

\newcommand\upAlg\Uparrow
\def\le{{\ell}}
\def\ex{e}

\def\lb{\left\llbracket}
\def\rb{\right\rrbracket}

\def\<{\left <}

\def\seq{\succcurlyeq}
\def\>{\right >}

\begin{document}

\title{Well-orders in the transfinite Japaridze algebra}

\author{David Fern\'andez-Duque,\\
Department of Mathematics,\\
 Instituto Tecnol\'ogico Aut\'onomo de M\'exico, \\
 david.fernandez@itam.mx\\
 \ \\
 Joost J. Joosten, \\
 Department of Logic, History and Philosophy of Science,\\ University of Barcelona,
 \\ jjoosten@ub.edu}  
 
\maketitle

\begin{abstract}

This paper studies the transfinite propositional provability logics $\glp_\Lambda$ and their corresponding algebras. These logics have for each ordinal $\xi< \Lambda$ a modality $\la \alpha \ra$. We will focus on the closed fragment of $\glp_\Lambda$ (i.e., where no propositional variables occur) and \emph{worms} therein. Worms are iterated consistency expressions of the form $\la \xi_n\ra \ldots \la \xi_1 \ra \top$. Beklemishev has defined well-orderings $<_\xi$ on worms whose modalities are all at least $\xi$ and presented a calculus to compute the respective order-types. 

In the current paper we present a generalization of the original $<_\xi$ orderings and provide a calculus for the corresponding generalized order-types $o_\xi$. Our calculus is based on so-called {\em hyperations} which are transfinite iterations of normal functions. 

Finally, we give two different characterizations of those sequences of ordinals which are of the form $\la {\formerOmega}_\xi (A) \ra_{\xi \in \ord}$ for some worm $A$. One of these characterizations is in terms of a second kind of transfinite iteration called {\em cohyperation.}
\end{abstract}

\section{Introduction}
%

In this paper we study transfinite propositional provability logics $\glp_\Lambda$ and their corresponding algebras. For an ordinal $\Lambda$, the transfinite provability logic $\glp_\Lambda$ is a polymodal version of G\"odel-L\"ob's provability logic \gl where for each ordinal $\alpha<\Lambda$ the logic contains a modality $[\alpha]$. 

These logics have been studied quite intensively lately and possess a very rich structure in various aspects. To mention just a few, it is a natural example of a logic that is not complete for its Kripke se\-man\-tics but is complete for its class of topological models \cite{BekBezhIcard2010, BeklemishevGabelaia:2011:TopologicalCompletenessGLP, Fernandez:2012:TopologicalCompleteness}. However, for natural topologies on intervals of ordinals the completeness for these spaces is independent of \zfc giving rise to various interesting set-theoretical questions (\cite{BagariaMagidorSakai:2013,Beklemishev:2011:OrdinalCompleteness,Blass:1990:InfinitaryCombinatorics}).

By $\glp$ we denote the class-sized logic that extends all $\glp_\Lambda$. In this paper we shall focus on $\glp^0$ --the closed fragment-- of this class-size logic. This is the fragment that does not contain any propositional variables hence is generated by $\top$, Boolean connectives and modalities only. Within $\glp^0$ we consider the class $\Worms$ of so-called {\em worms.} These are iterated consistency statements, that is, expressions of the form $\la \alpha _n \ra \ldots \la \alpha_1 \ra \top$. By $\Worms_\alpha$ we denote the class of worms where each occurring modality is at least $\alpha$.

In \cite{Beklemishev:2005:VeblenInGLP, BeklemishevFernandezJoosten:2012:LinearlyOrderedGLP} it has been shown that $\Worms_\alpha$ can be well-ordered by defining $A<_\alpha B :\Leftrightarrow \glp \vdash B \to \la \alpha \ra A$. For $A\in \Worms_\alpha$, by $\oOld_\alpha(A)$ we denote the order-type of $A$ in $\la \Worms_\alpha , <_\alpha \ra$. It is most natural to consider these well-orders as sub-structures of the algebras that correspond to $\glp$ which are often called \emph{Japaridze algebras}.

In this paper we study the ordering $<_\alpha$ as an ordering on {\em all} of $\Worms$. We will see that $<_\alpha$ no longer defines a linear order on $\Worms$; however, we prove that it does define a well-founded relation and denote the corresponding order-types by ${\formerOmega}_\alpha (A)$. We show how the ${\formerOmega}$ order-types can be recursively reduced to the $\oOld$ order-types, and in fact $o_\xi(A)=\oOld_\xi (A)$ whenever $A\in\Worms_\xi$. Based on this reduction we are able to give a calculus for the ordinal sequences $\vec{\formerOmega}(A) := \la {\formerOmega}_\xi (A)\ra_{\xi \in \ord}$. That is, we show how to compute $\vec{\formerOmega}(A)$ for a given worm $A$ and prove which ordinal sequences are attained as $\vec{\formerOmega}(A)$ for some $A$. The calculus we give is based on {\em hyperexponentials} and {\em hyperlogarithms,} which are operations on ordinals related to Veblen progressions and presented in detail in \cite{FernandezJoosten:2012:Hyperations}.

Our calculus for $o_\xi$ is different from the $\oOld_\xi$ calculus as presented in \cite{Beklemishev:2005:VeblenInGLP} in at least three essential ways. First, the definition of $o_\xi$ is a genuine generalization of $\oOld_\xi$ and $\oOld_\xi$ can be obtained a special case. Second, our presentation does not work with normal forms on worms, either in the presentation of the calculus or in any of the proofs. Finally, our calculus uses hyperexponentials  whereas the calculus in \cite{Beklemishev:2005:VeblenInGLP} used Veblen functions. 

{It is known that the sequences $\vec o(A)$ can be interpreted proof-theoretically. In particular, $\glp_\omega$ has been used to perform a $\Pi^0_1$-ordinal analysis of Peano Arithmetic (\pa) and related systems (\cite{Beklemishev:2004:ProvabilityAlgebrasAndOrdinals}).  Meanwhile, it has been shown in \cite{FernandezJoosten:2012:TuringProgressions} that there exists a close relation between Turing progressions of first-order theories and the sequences $\vec {\formerOmega} (A)$. There are ongoing efforts to carry these techniques to stronger theories using transfinite provability operators \cite{Beklemishev:2013:PositiveProvabilityLogic, Dashkov:2012:PositiveFragment, FernandezJoosten:2013:OmegaRuleInterpretationGLP, Joosten:2013:AnalysisBeyondFO} 

Furthermore, in \cite{FernandezJoosten:2012:WellOrders2}} it is discussed how the sequences $\vec {\formerOmega} (A)$ unveil important information about Kripke and other semantics for the closed fragment of $\glp_\Lambda$ as presented in \cite{FernandezJoosten:2012:KripkeSemanticsGLP0, FernandezJoosten:2012:ModelsOfGLP}.

\paragraph{Layout.} Section \ref{SecPrelim} introduces the logics $\glp_\Lambda$ and their fragments. In Sections \ref{section:wormsAndOrdinals} and \ref{section:wormsAndOrdertypes} we present the linear orders $<_\alpha$ and their corresponding order-types $o_\xi$ on substructures of the Japaridze algebra which are a central focus of this paper. We show how the computation of $o_\xi$ can be reduced to the computation of $o_0$.

In Section \ref{section:CalculusForO}, we give a calculus to compute $o_0$. The calculus that we present is actually a reduction to what we call \emph{worm enumerators} $\sigma^\xi$. It is in Section \ref{section:ComputingHyperexponentials} where we provide a calculus to compute the worm enumerators $\sigma^\xi$. Section \ref{section:OrdinalArithmetic} reviews the notions of {\em hyperations} and {\em cohyperations} from (\cite{FernandezJoosten:2012:Hyperations}) to show how the worm enumerators $\sigma^\xi$ are hyperations of ordinal exponentiation. Finally, in Section \ref{section:consistencySequences} we set the hyperations and cohyperations to work to obtain full characterizations of the sequences $\la o_\xi (A) \ra_{\xi \in {\sf On}}$.

The current paper is based on material which originally appeared in the unpublished manuscripts \cite{FernandezJoosten:2012:WellOrders:Version1} and \cite{FernandezJoosten:2012:WellOrders2}. Portions of the latter were reported in \cite{FernandezJoosten:2012:TuringProgressions}.

\section{Provability logics and the Reflection Calculus}\label{SecPrelim}


In this section we introduce the logics $\glp_\Lambda$ and its fragments, as well as fixing some notation.

\subsection{The logics $\glp_{\Lambda}$}
The language of $\glp_{\Lambda}$ is that of propositional modal logic that contains for each $\alpha < \Lambda$ a unary modal operator $[\alpha]$.
In the definition below the $\alpha$ and $\beta$ range over ordinals and the $\psi$ and $\chi$ over formulas in the language of $\glp_{\Lambda}$. 

\begin{definition}
For $\Lambda$ an ordinal, the logic $\glp_{\Lambda}$ is the propositional normal modal logic that has for each $\alpha < \Lambda$ a modality $[\alpha ]$ and is axiomatized by all propositional logical tautologies together witht the following schemata:
\[
\begin{array}{ll}
[\alpha] (\chi \to \psi) \to ([\alpha] \chi \to  [\alpha]\psi), & \\
{}[ \alpha ] ([\alpha] \chi \to \chi) \to [\alpha] \chi, &\\
\langle \alpha \rangle \psi \to [\beta] \langle \alpha \rangle \psi & \mbox{for  $\alpha < \beta$,}\\
{}[\alpha] \psi \to [\beta] \psi &\mbox{for  $\alpha \leq \beta$}. 
\end{array}
\]
The rules of inference are Modus Ponens and necessitation for each modality: $\frac{\psi}{[\alpha]\psi}$.
By \glp we denote the class-size logic that has a modality $[\alpha]$ for each ordinal $\alpha$ and all the corresponding axioms and rules. The classic G\"odel-L\"ob provability logic \gl is denoted by $\glp_1$.
\end{definition}

\subsection{Japaridze algebras}

The relations $<_\alpha$ do not give proper linear orders on $\Worms_\alpha$, given that different worms may be equivalent in \glp and hence undistinguishable in the ordering. We remedy this by passing to the Lindenbaum algebra of $\mathsf{GLP}$ -- that is, the quotient of the language of $\mathsf{GLP}$ modulo provable equivalence.

This algebra is a {\em Japaridze algebra}, as described below:
\begin{definition}[Japaridze algebra]
A {\em Japaridze algebra} is a structure \[\mathcal J=\langle D,\{[ \alpha]\}_{\alpha<\Lambda},\wedge,\neg,0,1\rangle\]
such that
\begin{enumerate}
\item $\langle D,\wedge,\neg,0,1\rangle$ is a Boolean algebra,
\item $[\alpha]1=1$ for all $\alpha<\Lambda$,
\item $[\alpha](x\to y)\leq [\alpha]x\to[\alpha]y$ for all $\alpha<\Lambda$, $x,y\in D$,
\item $[\alpha]([\alpha]x\to x)\leq [\alpha]x$ for all $\alpha<\Lambda$, $x\in D$,
\item $[\alpha] x\leq [\beta] x$ for all $\alpha\leq\beta<\Lambda$, $x\in D$ and,
\item $\langle\alpha\rangle x\leq [\beta]\langle \alpha\rangle x$ for all $\alpha<\beta<\Lambda$, $x\in D$,
\end{enumerate}
where $\langle\alpha\rangle,\to$ are defined in the usual way.
\end{definition}

It is in these algebras that the partial orders $<_\alpha$ we have described naturally reside. However, as we are mainly  interested in formulas that fall within a specific fragment of our language, we will work throughout the paper in a restricted calculus.

\subsection{The Reflection Calculus}

In \cite{Dashkov:2012:PositiveFragment, Beklemishev:2013:PositiveProvabilityLogic, Beklemishev:2012:CalibratingProvabilityLogic} Dashkov and Beklemishev introduced a calculus for reasoning about a fragment of the language of $\glp$ and called it the {\em Reflection Calculus} (\RC). (Closed) formulas of \RC are built from the grammar
\[\top \ | \ \phi\wedge\psi \ | \ \lambda\phi,\]
where $\lambda$ is an ordinal and $\phi,\psi$ are formulas of \RC; $\lambda\phi$ is interpreted as $\langle\lambda\rangle\phi$, but as \RC does not contain operators of the form $[\lambda]$, the brackets become unnecessary. \RC derives {\em sequents} of the form $\phi\vdash\psi$, given by the following rules and axioms:
\[\begin{array}{lcc}
\phi\vdash\phi&\phi\vdash\top&\dfrac{\phi\vdash\psi \ \ \psi\vdash \chi}{\phi\vdash\chi}\\\\
\phi\wedge\psi\vdash\phi&\phi\wedge\psi\vdash\psi
&\dfrac{\phi\vdash\psi \ \ \phi\vdash \chi}{\phi\vdash\psi\wedge\chi}\\\\
\alpha\alpha\phi\vdash\alpha\phi&\dfrac{\phi\vdash \psi}{\alpha\phi\vdash\alpha\psi}&\\\\
\beta\phi\vdash\alpha\phi \ \ & \beta\phi\wedge\alpha\psi\vdash \beta(\phi\wedge\alpha\psi)&\text{for $\alpha<\beta$.}
\end{array}
\]

In the context of \glp we shall sometimes denote $\glp \vdash \phi \to \psi$ by $\phi \vdash \psi$. The following is proven in \cite{Beklemishev:2013:PositiveProvabilityLogic}:

\begin{theorem}
$\glp$ is a conservative extension of \RC.
\end{theorem}

This result implies that any reasoning carried out within $\glp$ can, in principle, be carried out within \RC, and we shall use this calculus in all formal reasoning in this paper. As such will write $\lambda\phi$ instead of $\langle\lambda\rangle\phi$, unless the brackets are needed for legibility. We will merely write $\phi\vdash\psi$ to mean ``$\phi\vdash\psi$ is a theorem of $\RC$'', and for formulas of \RC, we will write $\phi\equiv\psi$ if $\phi\vdash \psi$ and $\psi\vdash\phi$. The equivalence class of $\phi$ under $\equiv$ will be denoted $\overline \phi$. For a set of formulas $\Phi$, we denote by $\overline{\Phi}$ the set of equivalence classes of its elements. 

\subsection{Worms and the closed fragment}

A closed formula in the language of \glp is simply a formula without propositional variables. In other words, closed formulas are generated by just $\top$ and the Boolean and modal operators. 

The closed fragment of \glp is the class of closed formulas provable in \glp and is denoted by $\glp^0$.
Within this closed fragment and the corresponding algebra, there is a particular class of privileged inhabitants/terms which are called \emph{worms}. Worms are nothing more than iterated consistency statements.

\begin{definition}[Worms, $\Worms$, $\Worms_\alpha$]
By $\Worms$ we denote the class of \emph{worms} of \glp which is inductively defined as $\top \in \Worms$ and $A\in \Worms \Rightarrow \langle \alpha \rangle A \in \Worms$. Note that every worm belongs to the language of \RC.

Similarly, we inductively define for each ordinal $\alpha$ the class of worms $\Worms_{\alpha}$ where all ordinals are at least $\alpha$ as $\top \in \Worms_{\alpha}$ and $A\in \Worms_{\alpha} \wedge \beta \geq \alpha \Rightarrow \langle \beta \rangle A \in \Worms_\alpha$.
\end{definition}

Both the closed fragment of \glp and the set of worms have been studied  in \cite{Beklemishev:2005:VeblenInGLP} and \cite{BeklemishevFernandezJoosten:2012:LinearlyOrderedGLP}. 
Worms can be conceived as the backbone of $\glp^0$ and obtain their name from the heroic worm-battle, a variant of the Hydra battle (see \cite{Beklemishev:2006}). 

\begin{notation}
We reserve lower-case letters at the beginning of the Greek alphabet $\alpha, \beta, \gamma, \hdots$ for variables ranging over ordinals. Also $\xi$ and $\zeta$ will denote ordinals. Worms will be denoted by upper case latin letters $A, B, C,  \hdots$. The Greek lower-case letters $\phi, \psi, \chi, \hdots$ will denote formulas. However, $\varphi$ shall be reserved for the Veblen enumeration function. Likewise, we reserve $\omega$ to denote the first infinite ordinal.
\end{notation}

By $|A|$, the length of a worm $A$, we shall mean the number of modalities occurring in $A$: $|\top| = 0$, and $|\la \xi \ra A| = |A| +1$. For $A$ a worm and $n$ a natural number we define the $n$-times concatenation of $A$ --denoted by $A^n$-- as usual: $A^0 = \top$ and $A^{n+1} =  A  A^n$. We will denote concatenation of worms by juxtaposition, defined recursively so that $\top A=A$ and $(\xi B) A=\xi(BA)$.


\section{Ordering worms}\label{section:wormsAndOrdinals}

In this section we define various natural ordering on worms and see how these orderings are related to each other.

\subsection{Worms and consistency orderings}

It is a fact of experience that natural mathematical theories can be linearly ordered in terms of consistency strength. Something similar holds for worms which motivates the next definition. 

\begin{definition}[$<, <_{\xi}$]
We define a relation $<_{\xi}$ on $\Worms \times \Worms$ by 
\[
A <_{\xi} B \ :\Leftrightarrow \  B \vdash \langle \xi \rangle A .
\]
Instead of $<_0$ we shall simply write $<$.
\end{definition}

We shall sometimes refer to the orderings $<_\xi$ as the \emph{$\xi$-consistency orderings}, and these orderings and their corresponding order-types are the main theme of this paper.
It is known (\cite{Beklemishev:2005:VeblenInGLP, BeklemishevFernandezJoosten:2012:LinearlyOrderedGLP}) that the class of worms is linearly ordered by $<_0$; that is, if $A,B$ are worms then either $A<_0 B$, $A\equiv B$ or $B <_0 A$.

Recall that $\overline \Worms_\xi$ denotes the class of worms in $\Worms_\xi$ modulo \glp-provable equivalence. Clearly, $<_\xi$ is well-defined on any of the $\overline \Worms_\zeta$ by $\overline A <_\xi \overline B \Leftrightarrow A <_\xi B$ whence we shall use the same symbol $<_\xi$ for both relations. The following theorem is proven in (\cite{Beklemishev:2005:VeblenInGLP, BeklemishevFernandezJoosten:2012:LinearlyOrderedGLP}).

\begin{theorem}\label{theorem:wormsIsomorphicToOrdinals}
For each ordinal $\xi$, we have  that $\la \overline \Worms_\xi, <_\xi \ra$ is isomorphic to the class of all ordinals.
\end{theorem}

As a consequence we see that $<_\xi$ is irreflexive on $\Worms$. For, suppose that $A<_\xi A$ for some $A\in \Worms$, then $A\vdash \xi A \vdash 0 A$ contradicting the irreflexivity of $<_0$ on $\Worms_0 (= \Worms)$.

The existence of a minimal element and the fact that each element has a direct $<_\xi$ successor in $\la \overline \Worms_\xi, <_\xi \ra$ are reflected in the following easy lemma.

\begin{lemma}\label{theorem:ordersHaveMinimalElementAndAreDiscrete}\ 
\begin{enumerate}
\item
$\top$ is a $<_\xi$-minimal element; \label{theorem:ordersHaveMinimalElementAndAreDiscrete:minimalElement}
\item
For no worms $A,B$ do we have $A <_\xi B <_\xi \xi A$.\label{theorem:ordersHaveMinimalElementAndAreDiscrete:discrete}
\end{enumerate}
\end{lemma}

\begin{proof}
For the first item, suppose that for some $A$ we have $A<_\xi \top$, then $\top \vdash \la \xi \ra A \vdash \la \xi \ra \top$ whence $\top <_\xi \top$ which contradicts the irreflexivity of $<_\xi$.

For the second, suppose towards a contradiction that there were such a $B$. Then $\xi A \vdash \xi B \vdash \xi \xi A$ whence $\xi A <_\xi \xi A$ which once again contradicts the irreflexivity of $<_\xi$.
\end{proof}

The orderings $<_\alpha$ for any ordinal $\alpha>0$ are not linear on $\overline \Worms$. For example, $1$ and $101$ are $\alpha$ incompatible for $\alpha >0$: Suppose $101 \vdash \alpha 1$, then $101 \vdash 11 \vdash 11 \wedge 01 \vdash 1101 \vdash 0101$, i.e. $101 < 101$ which contradicts the irreflexivity of $<$. Likewise $1\vdash \alpha 101 \vdash 0101 \vdash 01$ yields a contradiction. Also $1\equiv 101$ contradicts reflexivity since $1\vdash 101 \vdash 001 \vdash 01$. Similarly, it is easy to construct infinite anti-chains --see \cite{FernandezJoosten:2012:WellOrders2} for examples-- for $<_\x$ when $\x > 0$, hence the $<_\xi$ orderings do not define a well-quasiorder on $\Worms$. 

The next two lemmata are folklore and follow easily from the axioms of \glp. They shall be used repeatedly often without explicit mention in the remainder of this paper.

\begin{lemma}\label{lemma:basicLemma}\ \\
\vspace{-0.4cm}
\begin{enumerate}
\item\label{lemma:basicLemma1}
Given formulas $\phi$ and $\psi$, if $\beta < \alpha$, then
$(\alpha  \phi \wedge  \beta  \psi) \equiv  \alpha (\phi \wedge  \beta  \psi)$;

\item\label{lemma:basicLemma2}
If $A\in \Worms_{\alpha+1}$, then $A \wedge \langle \alpha \rangle B \equiv A\alpha B$;

\end{enumerate}
\end{lemma}

\proof
The left-to-right direction of the first item is an axiom of \RC. For the other direction we observe that $ \alpha   \beta \psi \vdash  \beta \psi$ in virtue of axioms $\alpha  \beta  \psi \vdash  \beta \beta \psi$ and $\beta\beta\psi\to\beta\psi$. The second item follows directly from the first by induction on $|A|$. For more details, we refer to \cite{BeklemishevFernandezJoosten:2012:LinearlyOrderedGLP}.
\qed

The next lemma tells us that in various occasions we are allowed to substitute equivalent parts into worms.

\begin{lemma}\label{theorem:equivalenceLemma}\ \\
\vspace{-0.4cm}
\begin{enumerate}
\item\label{theorem:equivalenceLemma1}
If $A,B \in \Worms_{\alpha+1}$ and $A\equiv B$, then for any worm $C$ we have $A\alpha C \equiv B \alpha C$;

\item\label{theorem:equivalenceLemma2}
If $A,B,C \in \Worms$ and $B\equiv C$, then $AB \equiv AC$;

\item\label{theorem:equivalenceLemma3}
More generally, if $A,B,C \in \Worms$ and $B\vdash C$, then $AB \vdash AC$;

\item\label{theorem:equivalenceLemma4}
For $A,B \in \Worms$ we have $AB\vdash A.$

\end{enumerate}
\end{lemma}

\begin{proof}
Item \ref{theorem:equivalenceLemma1} follows directly from Lemma \ref{lemma:basicLemma}.\ref{lemma:basicLemma2}, Item \ref{theorem:equivalenceLemma3} follows from an easy induction on the length of $A$ and Item \ref{theorem:equivalenceLemma2} follows from Item \ref{theorem:equivalenceLemma3}. Also, Item \ref{theorem:equivalenceLemma4} follows from Item \ref{theorem:equivalenceLemma3} by taking $C=\top$.
\end{proof}

We are not allowed to substitute just any part of a worm. For example, let us see that $1 \equiv 10$ but $11\not \equiv 101$: From \ref{theorem:equivalenceLemma}.\ref{theorem:equivalenceLemma4} we see that $10 \vdash 1$. Conversely, $1 \vdash 1 \wedge 0 \vdash 10$ by Lemma \ref{lemma:basicLemma}.\ref{lemma:basicLemma2} so that $1\equiv 10$. However, if we assume that $11 \equiv 101$, then $101 \vdash 11 \vdash 11 \wedge 01 \vdash 1101 \vdash 0101$ whence $101 \vdash 0101$. But this is nothing but $101 < 101$ which contradicts the irreflexivity of $<$.

So, in general we are not allowed to substitute equivalent parts into the left-most side of a worm. Lemma \ref{theorem:equivalenceLemma}.\ref{theorem:equivalenceLemma1} gives an exception and in Corollary \ref{theorem:equivalentSubstitutionOnLeftSeparatedByComma} we will see another exception: when $A\equiv B$, then $A0C \equiv B0C$.

\subsection{Decomposing and manipulating worms}\label{SubsDecomp}

In studying worms, and in particular to perform inductive arguments on them it is often useful to decompose worms into smaller worms. In this subsection we will introduce such decompositions, which will appear throughout the text. We use $A\leq_0 B$ as a shorthand for $A<_0 B$ or $A\equiv B$. Recall that we use $|A|$ to denote the length of $A$.

\begin{defi}\label{DefHead}
Let $A$ be a worm. By $h_{\x}(A)$ we denote the \emph{${\x}$-head} of ${A}$. Recursively: $h_{\x}(\top)=\top$, $h_{\x}(\zeta{A})= \zeta h_{\x}({A})$ if $\zeta\geq {\x}$ and $h_{\x}(\zeta{A})= \top$ if $\zeta < {\x}$.

Likewise, by $r_{\x}({A})$ we denote the \emph{${\x}$-remainder} of $A$: $r_{\x}(\top)=\top$, $r_{\x}(\zeta{A})= r_{\x}({A})$ if $\zeta\geq {\x}$ and $r_{\x}(\zeta{A})= \zeta{A}$ if $\zeta < {\x}$.
\end{defi}

In words, $h_{\x}({A})$ corresponds to the largest initial part (reading from left to right) of ${A}$ such that all symbols in $h_{\x}({A})$ are at least ${\x}$ and $r_{\x}({A})$ is that part of ${A}$ that remains when removing its ${\x}$-head. We thus have ${A} = h_{\x}({A})  r_{\x}({A})$ for all ${\x}$ and ${A}$. 

Observe that
\begin{equation}\label{heads}
 h_{\x}({A})  r_{\x}({A})\ \ \equiv \ \  h_{\x}({A}) \wedge r_{\x}({A}),
\end{equation}
as the first symbol of $r_{\x}({A})$ is less than ${\x}$ and $h_{\x}({A}) \in \Worms_{\x}$ (see Lemma \ref{lemma:basicLemma}).

As a particularly useful instance, we will write simply $h,r$ instead of $h_1,r_1$.

\begin{definition}
Given a worm $A$, we define $\body A=B$ if $r(A)=0B$, and $\body A=\top$ otherwise. We call $b(A)$ the {\em body} of $A$.
\end{definition}

\begin{lemma}\label{LemmBodyDecomp}
Given a worm $A\not=\top$, we have that 
\begin{enumerate}
\item
$A\equiv h(A)\wedge 0\body A$;

\item
$|\body A|<|A|$;

\item
$B<_0 r(A)$ if and only if ($B\leq_0 b(A)$ and $r(A)\not=\top$).
\end{enumerate}
\end{lemma}

\proof
We first address the first item. If $b(A)\not=\top$ then we know $A\equiv h(A)\wedge r(A)= h(A)\wedge 0b(A)$, otherwise, since $A\neq \top$, we always have $A\vdash 0$ so $A\equiv h(A)\wedge 0=h(A)\wedge 0b(A)$. 

It is obvious that $b(A)$ is always shorter than $A$ so that only the last item of the lemma needs to be proven. 

For the $\Leftarrow$ direction, suppose that $r(A) \neq \top$. Then, $r(A)= 0b(A)$ whence $B\leq_0 b(A) <_0 0b(A)$ and $B<_0 r(A)$.

For the $\Rightarrow$ direction, from $B<_0 r(A)$ we get that $r(A) \neq \top$ (Lemma \ref{theorem:ordersHaveMinimalElementAndAreDiscrete}.\ref{theorem:ordersHaveMinimalElementAndAreDiscrete:minimalElement}) whence $B<0b(A)$. Since $b(A) < B < 0b(A)$ is not possible (Lemma \ref{theorem:ordersHaveMinimalElementAndAreDiscrete}.\ref{theorem:ordersHaveMinimalElementAndAreDiscrete:discrete}) we get the required $B\leq_0 b(A)$.
\endproof

The following lemma tells us that if for some worm $A$ the first element is at most zero, then any worm $A'$ equivalent to $A$ must also start with a first element that is at most zero.

\begin{lemma}\label{LemmNoHead}
If there exists a worm $B$ such that $A\equiv r(B)$, then $h(A)=\top$.
\end{lemma}

\proof
Assuming that $A\equiv r(B)$, we clearly have $A\vdash h(A)\wedge r(B)$. If $h(A)\not=\top$ then $h(A)\vdash 1$, and thus $A\vdash 1\wedge r(B)\vdash 1r(B)\vdash 0r(B)$. This contradicts the irreflexivity of $<_0$.
\endproof

By Theorem \ref{theorem:wormsIsomorphicToOrdinals} we knew that there is a close relation between worms and ordinals and the above lemma exhibits yet another ordinal feature: If we can write an ordinal $\alpha$ as $\alpha' + 1$, then any other way of writing $\alpha$ must necessarily end with a `$+1$' too. This analogy will be made more precise after proving Lemma \ref{LemmOrdBasic}.\\
\medskip

So far we have seen operations on worms that decompose them into different parts. Another very important manipulation on worms is a sort of translation where all modalities in a worm are shifted by a constant amount. 

As we shall see in the remainder of this paper, this shift preserves a lot of structure and can even be conceived of as a functor between different spaces. 

We will define a shift to the right and one to the left. In order to define the shift to the left we need to recall a very basic fact from ordinal arithmetic (see for example \cite{Pohlers:2009:PTBook}).

\begin{lemma}\label{theorem:BasicPropertiesOrdinalArithmetic}
If $\zeta {<} \xi$ are ordinals, there exists a unique $\eta$ such that $\zeta + \eta = \xi.$
\end{lemma}
We will denote this unique $\eta$ by $-\zeta + \xi$ and it is this operation that is used to define our shift on worms to the left. We are now ready to introduce the shift to the right which is an operation $\alpha\uparrow$ that in general promotes worms to worms with higher consistency strength. As a converse operation we introduce a demoting operator $\alpha \downarrow$ which can be viewed as our shift to the left.

\begin{definition}[$\alpha\uparrow$ and $\alpha \downarrow$]
Let $A$ be a worm and $\alpha$ an ordinal. By $\alpha \uparrow A$ we denote the worm that is obtained by simultaneously substituting each $\beta$ that occurs in $A$ by $\alpha + \beta$. 

Likewise, if $A \in \Worms_\alpha$ we denote by $\alpha \downarrow A$ the worm that is obtained by replacing simultaneously each $\beta$ in $A$ by $-\alpha+\beta$.
\end{definition}
Note that by Lemma \ref{theorem:BasicPropertiesOrdinalArithmetic}, the operation $\alpha \downarrow$ is well-defined on $\Worms_\alpha$.

\begin{lemma}\label{theorem:uparrowProperties}

For $\alpha, \beta, \gamma$ ordinals and worms $A,B$ we have:

\begin{enumerate}
\item
$\alpha \uparrow \beta < \alpha \uparrow \gamma \ \Leftrightarrow \ \beta < \gamma$,\label{item:uparrowOrderPreserving:theorem:uparrowProperties}

\item
$\alpha \uparrow \beta \geq \beta$,

\item\label{item:additivity:theorem:uparrowProperties}
$\alpha \uparrow (\beta \uparrow A) = (\alpha + \beta)\uparrow A$,

\item \label{item:upDownarrowIsIdentity:theorem:uparrowProperties}
$\alpha\downarrow(\beta\uparrow A)=(-\alpha+\beta)\uparrow A$, provided $\alpha\leq \beta$,

\item\label{coadd}
$\alpha\downarrow(\beta\downarrow A)=(\beta+\alpha)\downarrow A$, provided $A\in \Worms_{\beta+\alpha}$,

\item\label{inverse}
$\alpha \uparrow ((\beta+\alpha) \downarrow A) = \beta \downarrow A$ for $A\in \Worms_{\beta+\alpha}$,


%
%
%


\end{enumerate}

\end{lemma}

\proof
The first three items are trivial. It is clearly sufficient to prove items \ref{item:upDownarrowIsIdentity:theorem:uparrowProperties} --- \ref{inverse} only for ordinals rather than for worms. All these items have similar elementary proofs. We shall prove Item \ref{item:upDownarrowIsIdentity:theorem:uparrowProperties} as an illustration. Thus, let $\alpha\leq \beta$ and fix some ordinal $\gamma$. We see that 
\[
\begin{array}{lll}
\alpha + (\alpha \downarrow \beta)\uparrow \gamma & = &\alpha + ((\alpha \downarrow \beta) + \gamma)\\
&  = & (\alpha + (\alpha \downarrow \beta)) + \gamma\\
  & = & \beta + \gamma.\\
\end{array}
\]
Thus, $(\alpha \downarrow \beta)\uparrow \gamma$ is the unique ordinal $\delta$ so that $\alpha + \delta = \beta + \gamma$. In other words, \\
$\alpha\downarrow(\beta\uparrow \gamma)=(-\alpha+\beta)\uparrow \gamma$, provided $\alpha\leq \beta$.
\qed

\medskip

As announced before, the shift operators preserve important structure as is expressed in the following lemma.

\begin{lemma}\label{theorem:uparrowInvarianceProperties} For worms $A,B \in \Worms_\xi$ we have
\begin{enumerate}

\item\label{item:SmallerRelationExtendsLargerRelation:theorem:uparrowInvarianceProperties}
$A <_\xi B \ \Leftrightarrow \ A<B$;

\item\label{item:comparingDifferentOrders:theorem:uparrowInvarianceProperties}
$A <_{\xi} B \ \Leftrightarrow \zeta \uparrow A <_{\zeta + \xi}  \zeta \uparrow B$.


\end{enumerate}

\end{lemma}

\begin{proof}

The $\Rightarrow$ direction of \ref{item:SmallerRelationExtendsLargerRelation:theorem:uparrowInvarianceProperties}
is easy. The other direction follows directly from the $\Rightarrow$ direction using irreflexivity and the fact that $<_{\xi}$ linearly orders $\overline \Worms_\xi$.

The $\Rightarrow$ direction of \ref{item:comparingDifferentOrders:theorem:uparrowInvarianceProperties} is the consequence of a more general observation. One can easily extend the operation $\zeta \uparrow$ to any formula of \RC. As the operation $\zeta \uparrow$ is order preserving on the ordinals one can easily prove by induction that any proof in \RC remains a proof after applying $\zeta \uparrow$ to every formula appearing in it. Thus, if $\phi\vdash \psi$, then also $\zeta\uparrow\phi \vdash \zeta \uparrow \psi$.

The $\Leftarrow$ direction follows directly from the $\Rightarrow$ direction using irreflexivity and the fact that $<_\xi$ is a linear order on $\overline\Worms_\xi$.
\end{proof}

As a consequence of this lemma we see that we can view each $\alpha {\uparrow}$ as an isomorphism between structures.

\begin{lemma}\label{theorem:IsomorphicWormFragments}
The map $\alpha {\uparrow}$ is an isomorphism between $(\Worms, <)$ and $(\Worms_\alpha,<_{\alpha})$.

Moreover, the map $\alpha\upAlg:\overline \Worms\rightarrow \overline \Worms_\alpha$ given by $\alpha\upAlg\overline A=\overline{\alpha\uparrow A}$ is well-defined and also defines an isomorphism.
\end{lemma}

\begin{proof}
The first claim follows from Property \ref{item:comparingDifferentOrders:theorem:uparrowInvarianceProperties} of Lemma \ref{theorem:uparrowInvarianceProperties}. Note that by Property \ref{inverse} of Lemma \ref{theorem:uparrowProperties} we see that $\alpha \uparrow (\alpha \downarrow A) = A$ for $A \in \Worms_\alpha$ so that $\alpha{\uparrow}$ is clearly a bijection.

To check the second point it suffices to observe that if $A\equiv B$ then in view of the first claim,  $\alpha\uparrow A\equiv\alpha\uparrow B$, so that the map $\alpha\upAlg$ is well-defined.
\end{proof}

\subsection{Reducing the $\xi$-consistency orderings}

In this subsection we shall see that any question of the form $A<_\alpha B$ can be reduced in various ways to simpler questions, for example, to questions of the form $A'<B'$.

To do so, we first need a reduction lemma (first published in \cite{FernandezJoosten:2012:TuringProgressions}). Recall from Definition \ref{DefHead} that $h_\xi(A)$ is the largest initial segment of $A$ which lies in $\Worms_\xi$, while $r_\xi(A)$ is the rest/remainder of $A$ after removing $h_\xi (A)$.

\begin{lemma}\label{lemma:reducingGeneralOrderToSpecialOrder}
Let $A$ and $B$ be worms and $\xi$ an ordinal. Then, $A>_\xi B$ if and only if $h_{\x}(A) >_{\x} h_{\x}(B)$ and $A \vdash r_{\x}(B).$
\end{lemma}

\proof
$(\Rightarrow)$ Assume $A>_\xi B$, i.e., $A\vdash \langle \xi\rangle B$. By \eqref{heads}, $B \equiv h_{\x}(B) \wedge r_{\x}(B)$ whence $A \vdash  {\x} B \vdash  {\x}  (h_{\x}(B) \wedge r_{\x}(B)) \vdash {\x}h_{\x}(B) \wedge  \x  r_{\x}(B)\vdash  \x  r_{\x}(B) \vdash r_{\x}(B)$, since $r_\x (B)$ is either $\top$ or starts with a modality stricly below $\xi$.

It remains to show that $h_\x (A) >_\x h_\x (B)$. Again, we write $A \equiv h_{\x}(A) \wedge r_{\x}(A)$. As $h_{\x}(A), h_{\x}(B) \in \Worms_{\x}$ we know that either (i) $h_{\x}(A) \equiv h_{\x}(B)$, (ii) $h_{\x}(B) \vdash \x h_{\x}(A)$ or (iii) $h_{\x}(A)\vdash \x  h_{\x}(B)$ holds, so it suffices to discard cases (i) and (ii) under the assumption that $A \vdash\x  B$ whence $A \vdash  \x  h_{\x}(B)$. 

Suppose now $h_{\x}(A) \equiv h_{\x}(B)$. Then, $A\equiv h_\x (A) \wedge r_\x (A) \vdash  \x  h_\x (B) \wedge r_\x (A) \vdash  \x  h_\x (A) \wedge r_\x (A) \vdash  \x A$ which contradicts the  irreflexivity of $<_{\x}$.

By a similar argument, the assumption that $h_{\x}(B) \vdash\x h_{\x}(A)$ contradicts the irreflexivity of $<_{\x}$ and we conclude that $h_{\x}(A)\vdash \x  h_{\x}(B)$.

$(\Leftarrow)$ This is the easier direction. Assume that $h_{\x}(A) \vdash  {\x} h_{\x}(B)$ and $A \vdash r_{\x}(B)$. Then, $A \equiv h_{\x}(A) \wedge r_{\x}(A)\vdash  \x  h_{\x}(B) \wedge r_{\x}(B)\vdash \x ( h_{\x}(B) \wedge r_{\x}(B))\vdash  \x B.$
\qed

In the right-hand side of Lemma \ref{lemma:reducingGeneralOrderToSpecialOrder} we see that the first conjunct $h_{\x}(A) >_{\x} h_{\x}(B)$ is only referring to worms in $\Worms_\x$ and their $<_\x$ relations. The worm $r_{\x}(B)$ starts with a modality strictly less than $\x$ and thus  the second conjunct $A \vdash r_{\x}(B)$ of the lemma can be settled by calling recursively to the lemma once more. Thus,  Lemma \ref{lemma:reducingGeneralOrderToSpecialOrder} recursively reduces the evaluation of statements of the form $A<_{\x}B$ with $A,B\in\Worms$ to evaluation of statements of the form $A'<_{\x'}B'$ with $A',B'\in\Worms_{\x'}$.

Moreover, we know (by Lemma \ref{theorem:uparrowInvarianceProperties}) that
\[
\begin{array}{lll}
h_\xi (A) >_\xi h_\xi (B) & \Leftrightarrow & h_\xi (A) >_0 h_\xi (B)\\
& \Leftrightarrow & \xi \downarrow  h_\xi (A) >_0 \xi \downarrow h_\xi (B).
\end{array}
\]
Thus, Lemma \ref{lemma:reducingGeneralOrderToSpecialOrder} tells us that by recursion on $\xi$, any question about $B<_\xi A$ can be reduced to question about $B'<_0 A'$. 

Thus, we can reduce questions about any $<_\xi$ ordering to questions about the $<_0$ ordering. We shall now see that we reduce questions about the $<_0$ ordering even further in that we may restrict questions of the form $A<_0 B$ to the case where one of $A$ or $B$ is either $\top$ or of the form $0C$.

Before we prove this further reduction, we first need an elementary lemma that relates the notions $<$, $\leq$, and $\vdash$.

\begin{lemma}\label{LemmVdashLeq}\ 
\begin{enumerate}
\item
$(A\leq B < C) \ \vee \ (A< B \leq C) \ \Rightarrow \ A<C$;

\item
$C\leq A \ \Leftrightarrow \ C<0A$;

\item
If $B\vdash A$ then $A\leq _0 B$;

\item
$A\vdash r(B) \ \Leftrightarrow \ A\geq r(B)$.
\end{enumerate}
\end{lemma}

\proof
Easy and left to the reader.
\endproof

Note that we cannot reverse the implication of the third item since, for example, it is easy to check that $1<01$ but $01 \nvdash 1$. We shall use our previous lemma without explicit mention in the remainder of this paper.

\begin{lemma}\label{LemmOrd}\ 
\begin{enumerate}
\item $A<_0B$ if and only if one of the following occurs:\label{LemmOrdOne}
\begin{enumerate}
\item $A<_0 r(B)$ or\label{LemmOrdOneA}
\item $r(A)<_0 B$ and $h(A)<_0 h(B)$.\label{LemmOrdOneB}
\end{enumerate}
\item $A\leq_0B$ if and only if\label{LemmOrdTwo}
\begin{enumerate}
\item $A\leq_0 r(B)$ or\label{LemmOrdTwoA}
\item $r(A)\leq_0 B$ and $h(A)\leq_0 h(B)$.\label{LemmOrdTwoB}
\end{enumerate}

\end{enumerate}
\end{lemma}

\proof
We prove the first item and first focus on the $\Leftarrow$ direction omitting various 0-subscripts. In case $A<r(B)$ we get $B\vdash h(B) \wedge r(B) \vdash r(B)\vdash 0A$ so $A<B$.

So, now suppose that $r(A) < B$ and $h(A) < h(B)$. Clearly, $h(A) < h(B) \Leftrightarrow h(A) <_1 h(B)$, and since $0r(B) \vdash r(B)$ we get
\[
B \vdash h(B) \wedge B\vdash 1h(A) \wedge B\vdash 1h(A) \wedge 0r(A) \vdash 1h(A) \wedge r(A) \vdash 1h(A)r(A) \vdash 0A
\]
whence $A<B$.

For the $\Rightarrow$ direction, 
assume that $A<_0 B$; let us show that if \ref{LemmOrdOneA} fails, then \ref{LemmOrdOneB} holds. Clearly we have $r(A)<_0 B$ since $B \vdash 0A \vdash 0(h(A) \wedge r(A))\vdash 0r(A)$. We wish to see that $h(A)<h(B)$. Since by assumption \ref{LemmOrdOneA} fails, we have that $r(B)\leq_0 A$ whence $A\vdash r(B)$. Now suppose for a contradiction that $h(B) \leq h(A)$. In case $h(B)\equiv h(A)$ we get
\[
A\vdash h(A) \wedge A \vdash h(B) \wedge A \vdash h(B) \wedge r(B) \vdash B \vdash 0A,
\]
and in case $h(B)< h(A)$ then also $h(B)<_1 h(A)$ and we get
\[
A\vdash h(A) \wedge A \vdash 1h(B) \wedge A \vdash 1h(B) \wedge r(B) \vdash 1B \vdash 10A \vdash 00A \vdash 0A,
\]
contradicting the irreflexivity of $<$ so that $h(A) < h(B)$ as was to be shown.

The proof of Lemma \ref{LemmOrd}.\ref{LemmOrdTwo} is similar. \endproof

Note that when asking the question $A<_0B$ we may always assume that one of $A$ or $B$ contains a zero, since $A<_0 B \ \Leftrightarrow \ \alpha \downarrow A <_0 \alpha \downarrow B$ where $\alpha$ is the smallest ordinal appearing in $AB$. Thus, indeed, by induction on $|A| + |B|$ we see that this lemma provides a reduction of questions about $<_0$ to questions about $<_0$ where one of the arguments is either $\top$ or starts with a $0$.

\section{Worms and ordinals}\label{section:wormsAndOrdertypes}

In the previous section we introduced various orderings on the worms. In the current section we shall study the corresponding order-types.

\subsection{Well-founded orders and order types}

An important corollary to our reduction lemma, Lemma \ref{lemma:reducingGeneralOrderToSpecialOrder}, is that the $<_\alpha$ orders are well-founded.

\begin{corollary}\label{theorem:UnrestrictedOrderIsWellFounded}
The relation $<_\alpha$ on $\Worms \times \Worms$ is well-founded.
\end{corollary}

\begin{proof}
Any hypothetical infinite descending $<_\alpha$-chain $A_0>_\alpha A_1>_\alpha A_2>_\alpha \hdots$ in $\Worms$ yields a corresponding chain $h_\alpha(A_0)>_\alpha h_\alpha(A_1)>_\alpha h_\alpha(A_2)>_\alpha \hdots$ in $\Worms_\alpha$ by Lemma \ref{lemma:reducingGeneralOrderToSpecialOrder}. This contradicts the fact that $<_\alpha$ defines a well-order on $\overline W_\alpha$ (Theorem \ref{theorem:wormsIsomorphicToOrdinals}).
\end{proof}

In virtue of this well-foundedness we can assign a $\xi${\em -order} to each worm by bar-recursion. In this subsection we shall revisit classical theory on how this assignment can be made in the general setting of well-founded orders. In the next subsection we shall apply the general results to the orders induced by the $<_\xi$ orderings.

\begin{definition}\label{definition:OmegaOrder}
If $X$ is either a set or a class and $\prec$ is a well-founded relation on $X$, we define $o_\prec\colon X\to{\sf On}$ by
\[
{\formerOmega}_\prec(x)=\sup_{y\prec x}({\formerOmega}_\prec(y) +1),
\]
where $\sup\varnothing=0$.
\end{definition}

Note that $o_\prec$ is well-defined due to the assumption that $\prec$ is well-founded.
In the case that $\prec$ is in addition a linear order, $o_\prec$ is in fact an isomorphism onto an initial segment of the ordinals. To make this precise, we first need an important definition. We shall identify the ordinal $\alpha$ with the set $\{  \beta \mid \beta < \alpha\}$ of ordinals $\beta$ less than alpha. 

\begin{definition}
If $\langle X,\prec\rangle$ is a well-order, define
\[
{\sf ot}_\prec(X)=\sup_{x\in X}(o_\prec(x)+1),
\]
where possibly ${\sf ot}_\prec(X)=\sf On$ if $X$ is a proper class.
\end{definition}

To formulate our result we shall use the following notation
\[
X^{\mathop\prec x}=\{y\in X\mid y\prec x\}.
\]

\begin{lemma}\label{LemmLinearIso}
For $\langle X,\prec\rangle$ a well-order, we have that $o_\prec\colon \langle X,\prec\rangle\to \langle {\sf ot}_\prec(X),<\rangle$ is an isomorphism. Moreover, if $x\in X$, then $o_\prec(x)={\sf ot}_\prec(X^{\mathop\prec x})$.
\end{lemma}

\PreprintVersusPaper{\proof
The proof is easy and details can be found, e.g., in \cite{Jech:2002:SetTheory, Pohlers:2009:PTBook}. For the sake of a stand-alone presentation we give the argument here. From the definition of $o_\prec$ we see that $x\prec y$ implies that $o_\prec(x)<o_\prec(y)$, while if $o(x)<o(y)$, we cannot have $x\seq y$, hence by totality $x\prec y$. Thus $x\prec y$ if and only if $o_\prec(x)<o_\prec (y)$.

Clearly $o_\prec(X)$ is an unbounded subset of ${\sf ot}_\prec(X)$. It remains to show that $o_\prec(X)$ is an initial segment, that is, if $\alpha<\beta$ and $\alpha\not\in o_\prec(X)$, then $\beta\not\in o_\prec(X)$. To see this, suppose that $\alpha\not\in o_\prec(X)$ and work by induction on $\beta\geq \alpha$ to show that $\beta\not\in o_\prec(X)$. The base case follows from our assumption, so we may take $\beta>\alpha$. If $\beta=o_\prec(x)$, then since by definition $o_\prec(x) = \sup \{ o_\prec(y) +1 \mid y \prec x\}$, for all $\gamma\in [\alpha,\beta)$ there is $y\prec x$ with $o_\prec(y)+1>\gamma$. But then in particular, we can find $y$ for which we have that $\alpha\leq o_\prec(y)<\beta$, and by our induction hypothesis there can be no such $y$.

The last claim follows immediately by observing that the definitions of $o_\prec(x)$ and ${\sf ot}_\prec(X^{\mathop\prec x})$ yield the same ordinal.
\endproof}{The proof is easy and details can be found, e.g., in \cite{Jech:2002:SetTheory, Pohlers:2009:PTBook, FernandezJoosten:2012:WellOrders}.\endproof}

In general, while $o_\prec$ is defined for any well-founded relation, it has much nicer behavior when $\prec$ is linear. For example, both $o_\prec$ and $o_\prec^{-1}$ are continuous \PreprintVersusPaper{in the following sense:}{as expressed by the following easy lemma whose proof we omit.}

\begin{lemma}\label{theorem:oIsContinuous}
If $\langle X,\prec\rangle$ is a well-ordered set and $S\subset X$ is bounded (in $X$), then $S$ has a supremum and $o_\prec(\sup S) = \sup o_\prec(S)$. 

Likewise, $o_\prec^{-1}(\sup \Gamma) = \sup o_\prec^{-1}(\Gamma)$ for any set of ordinals $\Gamma$.
\end{lemma}

\PreprintVersusPaper{\begin{proof}
Since both $o_\prec$ is an isomorphism, it basically just renames the objects and the order. Consequently $o^{-1}_\prec$ is also an isomorphism and both preserve suprema. By way of exercise, let us spell out the details of this reasoning.

That $S$ has a supremum follows immediately from the fact that it is bounded and $\prec$ is well-ordered. Now, since $o_\prec$ is order-preserving, $\sup o_\prec(S) \leq o_\prec(\sup S)$. Suppose for a contradiction that $\sup o_\prec(S) < o_\prec(\sup S)$.  Since $o_\prec$ is a surjection onto ${\sf ot}_\prec(X)$, we can pick $x$ so that $o_\prec(x)= \sup o_\prec(S)$. Again, using that $o_\prec$ preserves the order we see that for each $y\in S$ we have $y\leq x< \sup S$ which contradicts that $\sup S$ is the {least} upper bound of $S$.

The same argument applies to establish the continuity of $o_\prec^{-1}$.
\end{proof}}{}

\subsection{Worms and their order-types}

We shall now see how the observations of the previous subsection apply to $\overline{\Worms}$. If $X=\Worms$ and $\mathord\prec=\mathord<_\xi$ for some ordinal $\xi$, we write $o_\xi$ instead of $o_{<_\xi}$, while if $X=\overline{\Worms}$ we will write $\oAlg_\xi$.

Instead of $o_0$ we shall just write $o$. When $S$ is a set or class we shall denote by $o_{\alpha}(S)$ the image of $S$ under $o_{\alpha}$. It is easy to see that, for any $A\in\Worms$, $\oAlg_\alpha(\bar A)=o_\alpha(A)$ and moreover $\oAlg_\alpha (\overline\Worms_\beta) = o_\alpha (\Worms_\beta)$ for any ordinals $\alpha, \beta$. We shall often refer to the $o_\xi$ as $\xi$-consistency order-types.

\begin{lemma}\label{theorem:oDefinesAnIsomorphism}
Given $A,B\in \Worms$, $A<B$ if and only if $o(A)<o(B)$, and the map $o:\Worms\to \ord$ is surjective. Moreover, the map $\oAlg: \ \langle\overline{\mathbb W},<_0\rangle \to \langle\ord,<\rangle$ is an isomorphism.
\end{lemma}

\begin{proof}
Since $\langle \overline{\Worms},<_0\rangle$ is a well-order, we have by Lemma \ref{LemmLinearIso} that $\oAlg\colon \overline{\Worms}\to{\sf ot}_{<_0}(\overline{\Worms})$ is an isomorphism, and moreover since for $A,B\in\Worms$ we have that $A<_0 B$ if and only if $\overline A<_0\overline B$, it also follows that $A<_0 B$ if and only if $o(A)<o(B)$.\PreprintVersusPaper{

It remains to show that $\oAlg(\overline{\Worms})$ is unbounded in $\sf On$.}{It remains to show that $\oAlg(\overline{\Worms})$ is unbounded in $\sf On$.} But an easy induction shows that $o(\langle\alpha\rangle)\geq \alpha$ for all $\alpha$\PreprintVersusPaper{: $o(\la 0 \ra \top) = 1\geq 0$ (by Lemma \ref{theorem:ordersHaveMinimalElementAndAreDiscrete}), and $o(\la \alpha \ra \top) = \sup \{ o(B)+1 \mid B< \la \alpha \ra \top \} \geq \sup \{ o(\la \beta \ra \top)+1 \mid \beta<  \alpha  \}$ since $\{ \la \beta \ra \top \mid \beta < \alpha \} \subseteq \{ B\in \Worms \mid B<A \}$. By the induction hypothesis we see $\sup \{ o(\la \beta \ra \top)+1 \mid \beta<  \alpha  \} \geq \sup\{ \beta +1 \mid \beta < \alpha\} = \alpha$, so that $o(\la \alpha \ra \top) \geq \alpha$.}{.}
\end{proof}

Some elementary properties of the mappings $o_\xi$ are readily proven.

\begin{lemma}\label{LemmOrdBasic}\ 
\begin{enumerate}
\item\label{item:zero:LemmOrdBasic}
$o_\xi (\top) = 0$;

\item \label{LemmOrdBasicSucc}
$o_\xi (\xi A) = o_\xi (A)+1$;

\item \label{LemmOrdBasicNaturals}
$o_\xi(\xi^n) = n$.

\end{enumerate}
\end{lemma}

\begin{proof}
The first item is clear  (by Lemma \ref{theorem:ordersHaveMinimalElementAndAreDiscrete}.\ref{theorem:ordersHaveMinimalElementAndAreDiscrete:minimalElement}) since $\{ A \mid A <_\xi \top \} = \varnothing$ and $\sup \varnothing =0$.

For the second item, since $A<_\x \x A$ clearly $o_\x (A) +1 \leq o_\x (\x A)$. By Lemma \ref{theorem:ordersHaveMinimalElementAndAreDiscrete}.\ref{theorem:ordersHaveMinimalElementAndAreDiscrete:discrete}, we see that $\{ B \mid B<_\x \x A \} = \{ B \mid B<_\x A \} \cup \{ A \}$ so that actually $o_\x (A) +1 = o_\x (\x A)$.

Finally, Item \ref{LemmOrdBasicNaturals} follows from the other two by induction on $n$.
\end{proof}
 
 Let us conclude this subsection by identifying the `limit worms' and the `successor worms'. As usual, by ${\sf Succ}$ and ${\sf Lim}$ we denote the class of successor respectively limit ordinals.
\begin{lemma}\label{LemmLimitWorm} For any worm $A$ we have
\begin{enumerate}
\item $o(A) = 0 \Leftrightarrow A = \top$;
\item $o(A) \in {\sf Succ} \Leftrightarrow  h(A) = \top \neq A\Leftrightarrow  A= 0A'$ for some worm $A'$;
\item $o(A) \in {\sf Lim}  \Leftrightarrow  h(A) \neq \top \Leftrightarrow  A\neq \top \ \& \ A= \displaystyle\sup_{B<A} B.$
\end{enumerate}
\end{lemma}

\proof
The $\Leftarrow$ direction of Item $(1)$ is just Lemma \ref{LemmOrdBasic}.\ref{item:zero:LemmOrdBasic}. For the $\Rightarrow$ direction we see that if $A\neq \top$, then $A\vdash 0\top$ so that $o(A) > 0$.

The $o(A) \in {\sf Succ} \Leftarrow h(A) =\top \neq A$ direction of $(2)$ is already given in Lemma \ref{LemmOrdBasic}.\ref{LemmOrdBasicSucc}. The other direction follows from (1) in case $A=\top$ so we assume that $h(A)\neq\top$. If $o(A) \in {\sf Succ}$, then by definition of $o$ we have $o(A) = o(B) + 1$ for some $B<A$. Thus, it suffices to show that 
\begin{equation}\label{equation:limitBehaviorOfWormsWithNonEmptyHead}
\mbox{if $h(A)\neq \top $, and $B<A$ then $0B<A$. }
\end{equation}
Indeed, $h(A)\vdash 1$, hence $A\vdash 1\wedge 0B\vdash 10B\vdash 00B$, and $0B<A$, as claimed. The statement that $(h(A) = \top \neq A)\ \Leftrightarrow \ ( A= 0A' \mbox{ for some worm $A'$})$ is a mere syntactical triviality.

The $ o(A) \in {\sf Lim} \ \Leftrightarrow \ h(A) \neq \top$ equivalence of Item $(3)$ clearly follows from the other two items. Thus it remains to show that $o(A)\in \sf Lim$ if and only if $A\not=\top$ and $A=\sup_{B<A} B$. First assume that $o(A)\in \sf Lim$; then clearly $A\not=\top\neq h(A)$ and, by \eqref{equation:limitBehaviorOfWormsWithNonEmptyHead} if $B<A$, then $o(A)>o(B)+1=o(0B)$ so by Lemma \ref{theorem:oDefinesAnIsomorphism} $B<0B<A$; it follows that $A=\sup_{B<A}B$, since $\{B\in\Worms:B<A\}$ can have no smaller upper bound than $A$. 

Similarly, if $A\not=\top$ and $A=\sup_{B<A}B$, then given $C<A$ we have that $0C<A$ (as we cannot have that $0C\equiv A$), and thus $o(C)+1<o(A)$, which means that $o(A)\in\sf Lim$.
\endproof

Thus we will say that $A$ is a successor worm whenever $A=0A'$ for some worm $A'$, or a limit worm $A$ whenever $h(A) \neq \top$. 

\PreprintVersusPaper{In particular, in virtue of the lemma we can recast our reduction lemma of the $<_0$ ordering, Lemma \ref{LemmOrd}, to the extent that questions about the $<_0$ ordering can be reduced to the case where one of the arguments is not a limit worm.}{}

\subsection{Reducing the $\xi$-order-types}

In our reduction lemma, Lemma \ref{lemma:reducingGeneralOrderToSpecialOrder}, we saw how questions about the $<_\xi$ orderings could be reduced. This reduction could be viewed as a reduction to the $<_\zeta$ orderings restricted to the respective $\Worms_\zeta$'s as well as a reduction to the $<_0$ ordering. In this subsection we shall see that we have similar reductions for the order-types.

Let us temporarily introduce new orderings $\oOld_\xi$ which shall turn out to be just the restriction of $o_\xi$ to $\Worms_\xi$. 

\begin{definition}
For $A \in \Worms_\xi$ we define $\oOld_\xi(A)=
o_{<_\xi\upharpoonright \Worms_\xi}(A)$.
\end{definition}

Since we know that $<_\xi$ linearly orders $\overline \Worms_\xi$ we have an alternative characterization of $\oOld_{\xi}(A)$.

\begin{lemma}\label{theorem:alternativeCharacterizationTildeO}
Given $A\in\Worms_\xi$, $\oOld_{\xi}(A)={\sf ot}_{<_\xi}\{ \overline B \in \overline  \Worms_\xi \mid B <_{\xi} A\}.$
\end{lemma}

\begin{proof}
By Theorem \ref{theorem:wormsIsomorphicToOrdinals} we know that $\la \overline \Worms_\xi, <_\xi \ra$ is well-ordered, thus by Lemma \ref{LemmLinearIso}, $\oOld_{\xi}\colon \overline{\Worms}_\xi\to {\sf ot}_{<_\xi}(\overline{\Worms}_\xi)$ is an isomorphism. It remains to check that ${\sf ot}_{<_\xi}(\overline{\Worms}_\xi)=\sf On$; but this readily follows by observing (by a straightforward induction) that $\oOld_\xi(\langle\xi+\alpha\rangle)\geq\alpha$.
\end{proof}

The next lemma tells us how the computation of $o_\xi$ can be reduced to computations of $\oOld_\xi$.

\begin{lemma}\label{theorem:OmegaReducesToO}
For any worm $A$ and ordinal $\x$ we have ${\formerOmega}_{\x}(A) =  o_{\x}h_{\x}(A) =   \oOld_{\x}h_{\x}(A)$. 
\end{lemma}

\proof

Recall that for a partial order $\langle X,\prec\rangle$ and $x\in X$, we defined $X^{\prec x}=\{ y \in X \mid y \prec x \}$. We will write $\Worms_\xi^{<_\xi A}$ instead of $(\Worms_\xi)^{<_\x A}$.

We first see that
\begin{equation}\label{EqHead}
\Worms_\xi^{<_\xi h_\xi A} =h_\x\Worms^{<_\x A}:=\{ h_\x (B) \mid B<_\x A \}.
\end{equation}
For the $\subseteq$ inclusion, we fix some $C\in \Worms_\x$ and observe that $h_\x (C) = C$. Now, if $C<_\x h_\x(A)$ then clearly $A\vdash h_\x(A) \wedge r_\x(A) \vdash h_\x(A)\vdash \la \x \ra C$ whence also $C<_\x A$.

The other direction follows directly from Lemma \ref{lemma:reducingGeneralOrderToSpecialOrder} since $B<_\x A$ implies $h_\x(B) <_\x h_\x A$ and clearly $h_\x(B)\in \Worms_\x$. 

Now that we have this equality we proceed by induction on $A$ and obtain
\[
\begin{array}{lrrll}
 &{\formerOmega}_\x (A) \ \ & := & \sup_{B <_\x A} &\big({\formerOmega}_\x (B) +1\big) \\
\text{\sc ih}_{(B<_\xi A)} & & = & \sup_{B <_\x A} &\big(o_\x h_\x (B) +1 \big)\\
{\mbox{by } \eqref{EqHead}} & & = & \sup_{C\in \Worms_\x ^{<_\x h_\x(A)}} &\big(o_\x (C) +1 \big)\\
 \text{\sc ih}_{(C<_\xi h_\xi(A)\leq_\xi A)} & & = & \sup_{C\in \Worms_\x ^{<_\x h_\x(A)}}& \big( \oOld_\x (h_\x(C)) +1 \big) \\
 & & = & \sup_{C\in \Worms_\x ^{<_\x h_\x(A)}}&\big( \oOld_\x (C) +1 \big)\\
  & &= & & \oOld_\x h_\x (A).
\end{array}
\]
Likewise,
\[
\begin{array}{lrrll}
 &{\formerOmega}_\x (A) \ \ & = &  \sup_{ C\in \Worms_\x ^{<_\x h_\x(A)}}&\big(o_\x (C) +1 \big)\\
\mathsmaller{(h_\x A=h_\x h_\x A)} &  &= & \sup_{C\in \Worms_\x ^{<_\x h_\x h_\x(A)}}& \big(o_\x (C) +1 \big)\\
{\mbox{by } \eqref{EqHead}}  & & = & \sup_{C <_\x h_\x(A)} &\big(o_\x (h_\x (C)) +1 \big)\\ 
\text{\sc ih}_{(C<_\xi h_\xi (A)\leq_\xi A)} &  &= & \sup_{C <_\x h_\x(A)} &\big( o_\x (C) +1 \big)\\
  &  &= && o_\x h_\x (A).
\end{array}
\]
\qed

Indeed, this lemma yields a reduction of questions about the orders $o_\xi$ defined on $\Worms$ to questions about the orders $\oOld_\xi$ which are defined on $\Worms_\xi$. In particular, we see that $\oOld_\xi$ is just the restriction of $o_\xi$ to $\Worms_\xi$.

\begin{corollary}\label{theorem:oOrderRestrictedToWormsXiIsTildeO}
For $A\in \Worms_\xi$ we have that $o_\xi(A) = \oOld_\xi (A)$.
\end{corollary}

\begin{proof}
Immediate from Lemma \ref{theorem:OmegaReducesToO} since $h_\xi (A) = A$ for $A \in \Worms_\xi$.
\end{proof}

We now also obtain an alternative characterization of $o_\xi (A)$: If $\langle X,\prec\rangle$ is a partially ordered set (or class), a {\em $\prec$-chain} in $X$ is any subset $\mathcal C$ of $X$ which is linearly ordered by $\prec$.  
We reserve $\mathcal C$ to denote chains. Given $x\in X$ we will write $\mathcal C\prec x$ if $x$ is a strict upper bound for $\mathcal C$, we say that $\mathcal C$ is a {\em chain below $x$.}

It is straightforward to check that if $\mathcal C\prec x$ then ${\sf ot}_\prec(\mathcal C)\leq o_\prec(x)$. However, it may be that ${\sf ot}_\prec(\mathcal C)$ is always much smaller than $o_\prec(x)$ \PreprintVersusPaper{}{as is expressed in the following lemma which is folklore}.

\begin{lemma}
Given an ordinal $\gamma$, there exists a partially ordered set $\langle X,\prec\rangle$ and $x\in X$ so that $o_\prec(x)=\gamma$ but every $\prec$-chain below $x$ is finite.
\end{lemma}

\PreprintVersusPaper{
\proof[Proof sketch.]
We proceed by induction on $\gamma$. If $\gamma=0$, we may take $X$ to be a singleton. If $\gamma=\gamma'+1$, by induction hypothesis there is $\langle X',\prec'\rangle$ such that every $\prec'$-chain is finite yet there is a unique minimal point $x'\in X'$ with $o_{\prec'}(x')=\gamma'$. We then obtain $\langle X,\prec\rangle$ by adding a new minimal point $x$ which is $\prec$-below $X'$. It is clear that $o(x)=\gamma$, and any $\prec$-chain is merely a $\prec'$-chain, perhaps with $x$ added, hence finite.

Meanwhile, if $\gamma\in \sf Lim$, we have by induction hypothesis for each $\xi<\gamma$ a partiall ordered set $\langle X_\xi,\prec_\xi\rangle$ each with a unique minimal point $x_\xi$ with $o_{\prec_\xi}(x_\xi)=\xi$ and such that all $\prec_\xi$-chains are finite. Without loss of generality we may assume the sets $X_\xi$ to be all disjoint. Then we pick a fresh point $x_\gamma\not\in\bigcup_{\xi<\gamma}X_\xi$ and let $X_\gamma=\{x_\gamma\}\cup\bigcup_{\xi<\gamma}X_\xi$, and let $\prec_\gamma$ respect all previous orderings between elements except that $x_\gamma$ is the new minimum. As before, it is straightforward to check that $o_{\prec_\gamma}(x_\gamma)=\gamma$, yet as before there can be no infinite $\prec_\gamma$-chains.
\endproof
}
Thus it may not be the case that $o_\prec(x)$ is `attained'. To be precise, given a partially ordered set $\langle X,\prec\rangle$ and $x\in X$, say that {\em $o_\prec(x)$ is attained} if there is a chain $\mathcal C\prec x$ such that ${\sf ot}_\prec (\mathcal C)=o_\prec(x)$.

\begin{theorem}
If $A\in \Worms$ and $\xi$ is any ordinal, then
\[
o_\xi(A) = \sup_{\mathcal C <_\xi A} {\sf ot}_\xi (\mathcal C).
\]
Moreover, $o_\xi(A)$ is attained in $\la \Worms_\xi, <_\xi \ra$.
\end{theorem} 

\begin{proof}
To prove that $o_\xi(A) = \sup_{\mathcal C <_\xi A} {\sf ot}_\xi (\mathcal C)$, clearly it suffices that $o_\xi(A)$ is attained in $\la \Worms_\xi, <_\xi \ra$. 

Let $A\in \Worms$ and define $\mathcal C=\Worms_\xi^{<_\xi h_\xi(A)}=\{B\in \Worms_\xi:B<_\xi h_\xi(A)\}$. By Theorem \ref{theorem:wormsIsomorphicToOrdinals}, $\Worms_\xi$ is well-ordered and hence $\mathcal C$ is a chain; clearly also $\mathcal C<_\xi A$. Meanwhile, by Lemma \ref{theorem:OmegaReducesToO}, we have that $o_\xi(A)=\oOld_\xi h_\xi(A)$; however, by Lemma \ref{LemmLinearIso}, $\oOld_\xi h_\xi(A)={\sf ot}_{<_\xi}\mathcal C$. We conclude that $\mathcal C$ is a $<_\xi$-chain below $A$ with $o_\xi(A)={\sf ot}_{<_\xi}\mathcal C$, and thus $o_\xi(A)$ is attained.\end{proof}

Just like the reduction lemma for the $<_\x$ orderings yielded a reduction to $<_0$, by an additional lemma we shall see that Lemma \ref{theorem:OmegaReducesToO} also provides a reduction of $o_\x$ to $o_0$.

\begin{lemma}\label{theorem:LoweringSalphaOrders}
For $A \in \Worms_\xi$ we have $o_{\xi}(A) = o(\xi \downarrow A)$.
\end{lemma}

\proof
By Lemma \ref{theorem:uparrowProperties} it suffices to prove that for any worm $A \in \Worms$ we have $o_{\xi}(\xi \uparrow A) = o(A)$. Now, since $\xi \uparrow B <_\xi \xi \uparrow A \ \Leftrightarrow \ B<A$ (this is Lemma \ref{theorem:uparrowInvarianceProperties}) and since each $B\in \Worms$ is equal to $\xi \downarrow(\xi \uparrow B)$, we have that for each worm $A$
\begin{equation}\label{EqPlus}
\{ B \in \Worms \mid B<A \} = \{ \xi \downarrow C \mid C \in \Worms_\xi \wedge C<_\xi \xi \uparrow A \}.
\end{equation}
We will now show by induction on $<_0$ that for any worm $A$ we have $o(A) = \oOld_\xi (\xi \uparrow A)$:
\[
\begin{array}{llrll}
 & o (A) & = & \sup_{B<A}& \big(o (B) + 1 \big) \\
\text{\sc ih}& &  = & \sup_{B<A}& \big(\oOld_\xi (\xi \uparrow B) + 1 \big) \\
\mbox{by \eqref{EqPlus}}\ \ \ \ \ \ \ \ &  & = & \sup_{C \in \Worms_\xi^{<_\xi \xi \uparrow A}}& \big(\oOld_\xi (\xi \uparrow (\xi \downarrow C)) + 1 \big) \\
 & & = & \sup_{C \in \Worms_\xi^{<_\xi \xi \uparrow A}} &\big( \oOld_\xi ( C) + 1 \big)\\
 & & = && \oOld_\xi (\xi \uparrow A).\\
\end{array}
\]
We conclude our proof using Lemma \ref{theorem:OmegaReducesToO} to see that 
\[
o (A) = \oOld_\xi (\xi \uparrow A)  = o_\xi (h_\xi(\xi \uparrow A)) = o_\xi (\xi \uparrow A).
\]
\endproof
The reduction of $o_\xi$ to $o$ follows by combining Lemma \ref{theorem:LoweringSalphaOrders} and Lemma \ref{theorem:OmegaReducesToO}.

\begin{corollary}\label{corollary:reductionXiOrdertypeToPlainOrdertype}
For any worm $A$ and ordinal $\xi$ we have $o_\xi (A) = o (\xi \downarrow h_\xi (A))$.
\end{corollary}

Our temporary definition of $\oOld$ will not be used further on in the paper. In previous papers on well-orders in the Japaridze algebra one used the definition $\oOld_\xi$ and denoted that by $o_\xi$. By Corollary \ref{theorem:oOrderRestrictedToWormsXiIsTildeO} we know that the definition of $o_\xi$ in this paper coincides with the old definition when restricted to $\Worms_\xi$. Thus, our notation for the new definition causes no rupture with the tradition yet merely generalizes the existing theory.


\section{A calculus for the consistency order-types}\label{section:CalculusForO}

In this section we show how to compute the order-types $o(A)$. Actually, we shall provide a calculus that reduces the computation of $o(A)$ to the computation of what we call \emph{worm enumerators}. The calculus will consist of three cases: the empty worm, worms containing a zero and non-empty worms that do not contain a zero. Recall that the definitions of $h(A),b(A),r(A)$ may be found in Section \ref{SubsDecomp}.

\subsection{Worms containing a zero}

For ordinals we do not have that $\xi < \zeta \ \Leftrightarrow \ \xi + \alpha < \zeta + \alpha$. However for addition on the right we do have that $\xi < \zeta \ \Leftrightarrow \ \alpha+ \xi  <  \alpha+ \zeta$. For worms we have something similar although now the left-most side of the worm is determinant:

\begin{lemma}\label{LemmaMonotonoe}
Given worms $A,B,C$, we have that $A0C<_0 B0C$ if and only if $A<_0 B$.
\end{lemma}
In order to simultaenously prove both implications, we will instead prove the equivalent claim that if $A,B,C,D$ are worms such that $h(C)=h(D)=\top$ (i.e., $C,D$ begin with a zero or are $\top$), then $AC<_0 BC$ implies that $AD<_0 BD$. We work by induction on $|A|+|B|$. Observe that if $A_1,A_0$ are worms with $h(A_0)=\top$, then $h(A_1A_0)=h(A_1)$ while $r(A_1A_0)=r(A_1)A_0$.

If $AC<_0 BC$, then by Lemma \ref{LemmOrd} we have that either (i) $AC<_0 r(BC)$, or (ii) $r(AC)<_0 BC$ and $h(AC)<_0 h(BC)$. If (i) holds, since $r(BC)=r(B)C$, we have that $C \leq_0 AC<_0r(B)C$, hence by irreflexivity $r(B)\not=\top$ (or else $C<_0 C$) and thus $r(B)=0\body B$ and $\body{BC}=\body BC$. Thus by Lemma \ref{LemmBodyDecomp}, $ AC \leq_0 \body {BC}=\body BC$. By linearity this is equivalent to $\body BC\not<_0AC$, so that by our induction hypothesis $\body BD\not<_0 AD$, i.e. $AD\leq_0 \body BD$. It follows that $r(B)D=0b(B)D>_0 b(B)D\geq_0 AD$, and thus $BD=h(B)r(B)D\vdash h(B)0AD\vdash 0AD,$ so that $AD <_0 BD$.

If (ii) holds, we first claim that $BD\vdash r(A)D$. If $r(A)=\top$ this is obvious, otherwise we note that $\body AC<_0AC<_0 BC$, so by the induction hypothesis $\body AD<_0 BD$ and $BD\vdash 0\body AD=r(A)D$. Thus, since $h(A)=h(AC)<_0 h(BC)=h(B)$, we also have $h(A)<_1 h(B)$ and get $BD\equiv h(B)\wedge BD\vdash 1h(A)\wedge r(A)D \equiv 1h(A)r(A)D\vdash 0AD,$ and $AD<_0 BD$.
\endproof

We have seen at the end of Lemma \ref{theorem:equivalenceLemma} that on the left side of a worm, one is not allowed to replace a part by any equivalent part. The next corollary tells us that if we have a zero, then that allows us such a substitution and as such the zero functions as a sort of buffer.

\begin{corollary}\label{theorem:equivalentSubstitutionOnLeftSeparatedByComma}
For worms $A,B$ and $C$ we have that $A \equiv B $ if and only if $A0C \equiv B0C.$\end{corollary}

\begin{proof}
Immediate from Lemma \ref{LemmaMonotonoe}.
\end{proof}

The following is an analogue of Lemma \ref{theorem:BasicPropertiesOrdinalArithmetic}; it says that, for worms, we have a form of subtraction. In this case, however, it becomes ``right subtraction''.

\begin{lemma}\label{LemmaDecompose}
$A<_0B$ if and only if there exists $C$ such that $B\equiv C0A$. 
\end{lemma}

\proof
One direction is trivial: if $B\equiv C0A$, then clearly $B\vdash 0A$.

So assume that $A<_0 B$. We shall prove by induction on $|B|$ that if $A<_0B$, then we can find a worm $C$ so that $B \equiv C0A$. 

We consider two cases depending on $A<_0 r(B)$ or $r(B)\leq_0 A$.

In case $A<_0 r(B)$ we must have $r(B)\not =\top$, so $r(B)=0\body B$, whence $A\leq_0 \body B$. If indeed $A\equiv \body B$ then we have $B=h(B)0\body B\equiv h(B)0A$. Otherwise, $A<_0 \body B$ so by induction hypothesis $\body B\equiv C'0 A$ for some worm $C'$. We set $C=h(B)0C'$ and readily see that $B\equiv C0A$.

In case $A\geq_0 r(B)$ we claim that $B\equiv h(B)0A$. Indeed since $A<_0B$, $B\vdash h(B)\wedge 0A\equiv h(B)0A$, while since $A\geq_0 r(B)$, we also have $h(B)0A\vdash h(B)\wedge 0A\vdash h(B)\wedge 0r(B)\vdash h(B)\wedge r(B)\equiv B$.
\endproof

The above lemmas suggest that concatenations of the form $A0B$ behave much like addition; the following result makes this precise.

\begin{lemma}\label{theorem:basicPropertyOrder:successors}
Given worms $A,B$, $o(A0B)=o(B)+1+o(A)$.
\end{lemma}

\proof
By induction on $<_0$. First let us show that $o(A0B)\leq o(B)+1+o(A)$. Suppose that $C<_0 A0B$. If $C\leq _0 B$, then we already have $o(C)\leq o(B)<o(B)+1+o(A)$. If $C>_0 B$, then by Lemma \ref{LemmaDecompose} we have that $C\equiv D0B$ for some $D$, and thus by the induction hypothesis we have $o(C)=o(B)+1+o(D)$. But by Lemma \ref{LemmaMonotonoe}, $D<_0 A$ so $o(D)<_0 o(A)$ and thus $o(B)+1+o(D)<o(B)+1+o(A)$.

We now will show that $ o(B)+1+o(A) \leq o(A0B)$. Note that if $A=\top$, then $o(A0B) = o(0B) = o(B)+1$ by Lemma \ref{LemmOrdBasic}. Thus, we assume that $A\not \equiv \top$ so that we can choose $\xi< o(A)$. Then, since $o$ is an isomorphism between $\la \overline \Worms, <_0\ra$ and the ordinals (Lemma \ref{theorem:oDefinesAnIsomorphism}), there is a worm $C<_0 A$ with $o(C)=\xi$. By Lemma \ref{LemmaMonotonoe} we have that $C0B<_0 A0B$, and by induction $o(C0B)=o(B)+1+o(C)=o(B)+1+\xi$. Since $\xi<o(A)$ was arbitrary, we conclude that $o(A0B)\geq o(B)+1+o(A)$.
\endproof

In this subsection we have dealt with worms that do contain a zero and could recursively compute their order-types. We shall reduce worms that do not contain a zero to worms that do contain a zero via the $\xi \uparrow$ and $\xi \downarrow$ mappings.

\subsection{A calculus using the worm enumerators $\sigma^\alpha$}

A key role in the larger calculus is reserved for the functions $\sigma^\alpha$ that enumerate the $<$-orders of worms in $\Worms_\alpha$ in increasing order. We shall prove sufficiently many structural properties of these functions $\sigma^\alpha$ so that we end up with a recursive calculus to compute them. 

Moreover, it shall turn out that the functions $\sigma^\alpha$ can be viewed as transfinite iterations of a certain ordinal exponentiation that we shall call \emph{hyperexponential} functions and which we shall discuss in  Section
\ref{section:hyperations}.

\begin{definition}[Worm enumerators $\sigma^{\alpha}$]
We define $\sigma^{\alpha}$ to be the function that enumerates $o(\Worms_\alpha)$ in increasing order.
\end{definition}

We shall first see how a calculus for $o$ can be reduced to a calculus for these functions $\sigma^\alpha$. The following nice lemma characterizes $\sigma^{\alpha}$ as a conjugate of the map $\alpha \upAlg$ on worms.

\begin{lemma}\label{theorem:ConjugationLemma}
$o(\Worms_\alpha)=\oAlg (\overline\Worms_\alpha)$ is enumerated in increasing order by $\oAlg\circ\alpha{\upAlg} \circ \oAlg^{-1}$, that is, 
\[
\sigma^{\alpha} = \oAlg\circ\alpha{\upAlg} \circ \oAlg^{-1}.
\]
\end{lemma}

\begin{proof}
In the proof we shall explicitly write $<_0$ for the ordering on worms and $<$ for the ordering on ordinals. Lemma \ref{theorem:oDefinesAnIsomorphism} told us that $\oAlg: \langle\overline\Worms, <_0 \rangle\cong \langle\ord,<\rangle$. Thus by Lemma \ref{theorem:uparrowInvarianceProperties}.\ref{item:SmallerRelationExtendsLargerRelation:theorem:uparrowInvarianceProperties}
we see that for $A, B \in \Worms_\alpha$
\begin{equation}\label{equation:isomorphismWormsAlphaAndOrdinals}
A <_{\alpha} B \ \Leftrightarrow \ A<_0 B \ \Leftrightarrow \ o(A) < o(B). 
\end{equation}
If we combine this with the fact that $o(\Worms_\alpha)$ is an unbounded class of ordinals, we see that an order-preserving enumeration of $o(\Worms_\alpha)$ is nothing but the unique isomorphism between $\la {\sf On}, < \ra$ and $\la o(\Worms_\alpha), <_\alpha\ra$.

We can reformulate \eqref{equation:isomorphismWormsAlphaAndOrdinals} as $
\oAlg \colon \la \overline\Worms_\alpha, <_\alpha \ra \cong \la o(\Worms_\alpha),<\ra.
$ We also have by Lemma \ref{theorem:IsomorphicWormFragments} that $
\alpha {\upAlg} \colon \la \overline\Worms, <_0 \ra  \cong \la \overline\Worms_\alpha,<_{\alpha}\ra .$ Once more using the fact that $
\oAlg^{-1}\colon \la \ord,< \ra \cong \la \overline\Worms,<_0\ra ,$ we see by composing these three isomorphisms that $\oAlg\circ \alpha {\upAlg} \circ \oAlg^{-1} \colon \la \ord,< \ra \cong \la o(\Worms_\alpha),<\ra$, whence $\sigma^\alpha = \oAlg\circ \alpha {\upAlg} \circ \oAlg^{-1}$.
\end{proof}

So seeing $\alpha\upAlg$ as an action of the ordinals on $\overline\Worms$, and $\sigma^\alpha$ as an action of the ordinals on the ordinals, the above tells us that the two actions are isomorphic. Let us draw a nice corollary from our lemma.

\begin{corollary}\label{theorem:orderAndArrows}
For any worm $A$, $o(\alpha \uparrow A) = \sigma^\alpha o(A)$
\end{corollary}

\begin{proof}
We have that
\[o(\alpha\uparrow A)=\oAlg(\overline{\alpha\uparrow A})=\oAlg(\alpha \upAlg \overline A) = \oAlg(\alpha \upAlg \oAlg^{-1} (o(A)))= \sigma^\alpha o(A).\]
\end{proof}

We may recast this by stating that the following diagram commutes:
\[
\xymatrix
{
\Worms\ar[d]^o\ar[r]^{\alpha\uparrow}&\Worms\ar[d]^o\\
\ord\ar[r]^{\sigma^{\alpha}}&\ord
}
\]
With this we now obtain a complete calculus for computing $o$ and $o_\alpha$ once we know how to compute the functions $\sigma^\alpha$.

\begin{theorem}\label{theorem:OrderTypeCalculus}
Given worms $A,B$ and an ordinal $\xi$,
\begin{enumerate}
\item
$o(\top)=0$;
\item\label{item:omegeSums:theorem:OrderTypeCalculus}
$o(A0B)=o(B)+1+o(A)$;
\item\label{exo}
$o({\x}\uparrow {A}) = \sigma^{\x} o({A})$.
\end{enumerate}
\end{theorem}

The calculus in this theorem looks efficient and elegant but we seem to be running in circles here. To compute $o$ we need to know how to compute the worm-enumerators $\sigma^\x$. The $\sigma^\x$ in their turn are defined in terms of $o$. \PreprintVersusPaper{However, close inspection reveals a strategy to break out of this circle for worms that only contain ordinals below $\omega$.

\begin{remark}[Strategy]\label{remark:generalStrategyForComputingO}
For very simple worms of the form $0^n$ with $n\in \omega$ we know that $o(0^n)=n$ (this is Lemma \ref{LemmOrdBasic}.\ref{LemmOrdBasicNaturals}). So, it would be good if with our calculus we could reduce any question about $o(A)$ to a combination of questions of order types of these simple worms of the form $0^n$, 

By repeatedly applying Lemma \ref{theorem:basicPropertyOrder:successors} we know that we can write $o(A)$ as a sum of $1$'s and $o(B_i)$'s with all the $B_i$'s in $\Worms_1$. An important first step is thus to reduce questions about these $o(B_i)$ to questions about $o(1{\downarrow}B_i)$. In particular, such a reduction will be embodied in the function $\sigma^1$.
\end{remark}

In the next section we shall see how this strategy is implemented.}{In the next section we shall see how we can break out of this circle and provide a stand-alone calculus for our worm-enumerators.}

\section{Computing the worm enumerators $\sigma^\xi$}\label{section:ComputingHyperexponentials}

In this section we shall see how the worm enumerators $\sigma^\alpha$ can be computed. We shall provide a recursive calculus in Theorem \ref{theorem:RecursiveHyperexponential}. 

\subsection{Worm enumerators: basic properties}

The first step in characterizing the worm enumerators we get almost for free and consists of determining the ordinal function $\sigma^0$.

\begin{lemma}\label{theorem:sigmaZeroIsTheIdentity}
The function $\sigma^0$ is the identity function on the ordinals.
\end{lemma}

\begin{proof}
Recall that by definition, $\sigma^0$ enumerates $o(\Worms_0)$ in increasing order. Since $\oAlg$ defines an isomorphism between $\overline \Worms$ and $\sf On$, we see that $o(\Worms_0) = o(\Worms) = {\sf On}$.  Evidently, $\sf On$ is enumerated by the identity function, so that $\sigma^0(\alpha) = \alpha$ for each ordinal $\alpha$.
\end{proof}

As a second step in characterizing the worm enumerators $\sigma^\alpha$, we shall prove that for each ordinal $\alpha$, the corresponding $\sigma^\alpha$ is a normal (both increasing and continuous) function. Since each $\sigma^\alpha$ by definition enumerates a class of ordinals \emph{in increasing order}, it is clear that each $\sigma^\alpha$ is increasing. 

So, next we need to see that each $\sigma^\alpha$ is continuous, that is, that computing $\sigma^\alpha$ commutes with taking suprema: if $\Delta$ is a set of ordinals then $\sigma^\alpha (\sup \Delta) = \sup\sigma^\alpha \Delta$. Since we know that $\sigma^{\alpha} = \oAlg\circ\alpha{\upAlg} \circ \oAlg^{-1}$ (this is Lemma \ref{theorem:ConjugationLemma}),  it suffices to prove that taking suprema commutes with all of $\oAlg$, $\oAlg^{-1}$, and $\upAlg$. For $\oAlg$ and $\oAlg^{-1}$ this has been established in Lemma \ref{theorem:oIsContinuous}; it remains to show that $\alpha{\uparrow}$ can be viewed as continuous on $\Worms$. We shall state continuity for `limit worms'. Recall from Lemma \ref{LemmLimitWorm} that $o(A) \in {\sf Lim}$ iff $h(A) \neq \top$ in which case $A = \sup_{B<A} B$.

\begin{lemma}\label{theorem:UparrowIsContinuous}
Given a worm $A$ with $h(A)\neq\top$ and an ordinal $\alpha>0$,  then 
\[
\begin{array}{ll}
\alpha\uparrow A = \alpha\uparrow (\displaystyle\sup_{B<A}B) =\displaystyle\sup_{B<A}\alpha\uparrow B & \mbox{  \ \ \ and,}\\
\alpha\upAlg A = \displaystyle\sup_{B<A}\alpha\upAlg B. & \\ 
\end{array}
\]
\end{lemma}

\begin{proof}
Clearly, the second item follows from the first one. 
In Lemma \ref{LemmLimitWorm} it is proven that $A = \sup_{B<A} B$ whenever $h(A) \neq \top$ so we will concentrate on showing $\alpha\uparrow A =\sup_{B<A}\alpha\uparrow B$.
 
From $B<A$ we get by Lemma \ref{theorem:uparrowInvarianceProperties} that $\alpha \uparrow B <_\alpha \alpha \uparrow A$ whence $\alpha \uparrow B < \alpha \uparrow A$ so that $\sup_{B<A} \alpha \uparrow B \leq \alpha \uparrow A$ is clear.

In order to show that $\alpha \uparrow A \leq \sup_{B<A} \alpha \uparrow B$ it suffices to prove that for any $C< \alpha\uparrow A$ we can find $B<A$ such that $C< \alpha\uparrow B$. This we prove by induction on $|C|$, with the base case ($C=\top$) being trivial so that we may write $C=\beta \uparrow (C_10C_0)$. Moreover we shall pick $C_0$ and $C_1$ in the unique way so that $C_1 \in \Worms_1$ (this includes the case $C_1 = \top$).

If $\beta <\alpha$, then clearly both $C$ and $\alpha\uparrow A$ belong to $\Worms_\beta$. Using Lemmas \ref{theorem:uparrowProperties} and \ref{theorem:uparrowInvarianceProperties}, we have $C_10C_0=\beta\downarrow C< \beta\downarrow(\alpha \uparrow A)=(-\beta+\alpha)\uparrow A$. Since $C_0,C_1\leq C_10C_0$, by the induction hypothesis for $|C_0|,|C_1|<|C|$ there are $B_0,B_1<A$ with $C_i<(-\beta+\alpha)\uparrow B_i$. Taking $B=\max\{B_0,B_1\}$ we see that for $i=0,1$, $(-\beta + \alpha) \uparrow B\geq (-\beta+\alpha)\uparrow B_i  > C_i$. Since we chose $C_1$ with $C_1 \in \Worms_1$, we also have $(-\beta + \alpha) \uparrow B >_1 C_1$ whence $(-\beta + \alpha) \uparrow B \vdash 1 C_1\wedge 0C_0\vdash 1C_10C_0 \vdash 0C_10C_0$, i.e., $C_10C_0 < (-\beta + \alpha) \uparrow B$. It follows that $C = \beta \uparrow (C_10C_0) < \beta \uparrow ((-\beta + \alpha) \uparrow B) = \alpha \uparrow B$, as was to be proven.

If $\beta\geq \alpha$ we get that $C,\alpha\uparrow A\in \Worms_\alpha$, so that from $C < \alpha \uparrow A$ we obtain $\alpha\downarrow C< \alpha\downarrow(\alpha\uparrow A)= A.$ Since $A$ is a limit, we have that $\alpha\downarrow C<0(\alpha\downarrow C)<A$, thus $C= \alpha \uparrow (\alpha\downarrow C) < \alpha \uparrow (0(\alpha \downarrow C))$, as was to be shown.
\end{proof}

Now that we have established that $\alpha{\uparrow}$ can be viewed as continuous on $\Worms$ we can prove that the $\sigma^\alpha$ are continuous ordinal functions. 

\begin{lemma}\label{theorem:hyperexponentiationYieldNormalFunctions}
Each $\sigma^\alpha$ is a normal function.
\end{lemma}

\begin{proof}
It is clear that $\sigma^\alpha$ is increasing so we only need to see that $\sigma^\alpha$ is continuous for each $\alpha$. 
%
%
For the continuity of $\sigma^\alpha$ we reason as follows. Let $\lambda \in {\sf Lim}$ so that for some worm $A$ with $h(A) \neq \top$ we have $\lambda = o(A)$.
\[\begin{array}{lcccccc}
\sigma^\alpha (\lambda) &=&  \sigma^\alpha o(A)
&=&  \oAlg  \, (\alpha \upAlg) \,  \oAlg^{-1} o(A)
&=&  \oAlg  \, (\alpha \upAlg) \,  \bar A \\
&  =&  \oAlg  \, (\alpha \upAlg) \,  \displaystyle\sup_{B<A}  \bar B   
&= & \oAlg  \, \displaystyle\sup_{B<A}  (\alpha \upAlg) \,   \bar B   
&  = & \displaystyle\sup_{B<A} \oAlg \,  (\alpha \upAlg)\,   \bar B\\
& =&  \displaystyle\sup_{B<A} \oAlg \,  (\alpha \upAlg)\, \oAlg^{-1} o(  B) 
  &=&  \displaystyle\sup_{B<A} \sigma^\alpha o(B) 
  &=& \displaystyle\sup_{\beta<\lambda} \sigma^\alpha \beta,
\end{array}\]
where we use Lemmas \ref{theorem:oIsContinuous} and \ref{theorem:UparrowIsContinuous} to commute $\sup$ with $\oAlg$ and $\alpha\Uparrow$, respectively.
\end{proof}
\noindent
As a next step in characterizing the $\sigma^\alpha$ functions we shall set out to determine $\sigma^1$. 
\PreprintVersusPaper{In order to analyze $\sigma^1$ let us see how the general strategy as outlined in Remark \ref{remark:generalStrategyForComputingO} gets reflected in some concrete examples:

\begin{enumerate}
\item
We first look at the $<_1$-first non-trivial worm in $\Worms_1$. It is not hard to prove by elementary methods that $1= \sup_{n\in \omega}0^n$. In this sense, $\la 0^n \ra_{n\in \omega}$ is a natural \emph{fundamental sequence} for the worm $1$. Since we know that $o(o^n)=n$ we see that 
\[
o(1) = o(\sup_{n\in \omega}0^n) = \sup_{n\in \omega} o (0^n) = \omega.
\]
\item
The $<_1$-second non-trivial worm in $\Worms_1$ is $11$. Again, by elementary methods we can prove that $11 = \sup_{n\in \omega} (10)^n1$ so that $\la (10)^n1 \ra_{n\in \omega}$ can be conceived as a natural fundamental sequence for the worm $11$. Now that we know $o(1) = \omega$, by Lemma \ref{theorem:basicPropertyOrder:successors} we know that $o(101) = o(1) + 1 + o(1) = \omega\cdot 2$. By repeating this argument we see that $o((10)^n1)= \omega\cdot (n+1)$ so that we obtain
\[
o(11) = o(\sup_{n\in \omega}(10)^n1) = \sup_{n\in \omega} o ((10)^n1) = \sup_{n\in \omega} \omega \cdot (n+1) = \omega^2.
\]

\item
For the $<_1$-third non-trivial $\Worms_1$ worm $111$ we find the fundamental sequence $(110)^n11$. Using the previously computed value of $o(11)$ we see that $o((110)^n11) = \omega^2\cdot(n+1)$ so that 
\[
o(111) = o(\sup_{n\in \omega}(110)^n11) = \sup_{n\in \omega} o ((110)^n11) = \sup_{n\in \omega} \omega^2 \cdot (n+1) = \omega^3.
\]
\end{enumerate} 
These three examples readily suggest an inductive argument with two important ingredients. First, each worm of the form $1A$ has a natural fundamental sequence and second, due to the particular uniform nature of these fundamental sequences, there seems to be an $\omega$ power of difference between $o(A)$ and $o(1{\uparrow}A)$. These ingredients are embodied below in Lemma \ref{theorem:FundamentalSequencesForWormsStartingWith1} and Lemma \ref{LemmSigmaOne} respectively.
}{We first look at the $<_1$-first non-trivial worm in $\Worms_1$. It is not hard to prove by elementary methods that $1= \sup_{n\in \omega}0^n$. In this sense, $\la 0^n \ra_{n\in \omega}$ is a natural \emph{fundamental sequence} for the worm $1$. Since we know that $o(o^n)=n$ we see that $o(1) = o(\sup_{n\in \omega}0^n) = \sup_{n\in \omega} o (0^n) = \omega$. 

In a similar fashion we see that for the $<_1$-second non-trivial worm in $\Worms_1$ which is $11$, can prove that $11 = \sup_{n\in \omega} (10)^n1$ so that $\la (10)^n1 \ra_{n\in \omega}$ can be conceived as a natural fundamental sequence for the worm $11$. Using that $o(1) = \omega$, by repeatedly applying Lemma \ref{theorem:basicPropertyOrder:successors} we know that $o((10)^n1)= \omega\cdot (n+1)$ so that $o(11) = o(\sup_{n\in \omega}(10)^n1) = \sup_{n\in \omega} o ((10)^n1) = \sup_{n\in \omega} \omega \cdot (n+1) = \omega^2
$. 

The following two lemmas are inspired by these examples and establish that limit worms admit uniform fundamental sequences and that there is essentially a power $\omega$ difference between $o(A)$ and $o(1{\uparrow}A)$.
}

\begin{lemma}\label{theorem:FundamentalSequencesForWormsStartingWith1}
Let $B\in \Worms_1$. For $A=1B$, we have that $A=\sup_{n<\omega}A[n],$ where $A[0]=B$ and $A[n+1]= B0A[n]$.
\end{lemma}

\proof
That $A>A[n]$ for all $n$ follows by induction; the base case is easy since $A= 1B \vdash 0 B$. For the induction step,
\[
A\stackrel{\rm IH}\vdash 1 B \wedge 0A[n] \equiv 1 B 0A[n] \vdash 0B 0A[n] = 0A[n+1].
\]

Meanwhile, we prove by induction on the length of $C$ that if $C< A$ then $C<A[n]$ for some $n$. We assume $C\not=\top$, otherwise the claim is trivial. Thus, from $\body C<C<A$, the fact that $|\body C|<|C|$ and the induction hypothesis we obtain $\body C<A[n]$ for some $n$ so that also $0A[n]\vdash 0\body C$. 

Then, from $h(C)\leq C<_1 A$ we obtain $h(C)\leq B$ whence $h(C)\leq_1 B$ and $B\vdash 1h(C)$. Thus,
\[
A[n+1] = B 0A[n] \equiv B \wedge 0A[n] \vdash 1h(C) \wedge 0\body C \equiv 1h(C) 0\body C \vdash 0h(C) 0\body C
\]
whence $C<A[n+1]$.
\endproof

With the use of this lemma, we can now establish a relation between $o(A)$ and $o(1{\uparrow} A)$, as to determine $\sigma^1$:

\begin{lemma}\label{LemmSigmaOne}\ 
\begin{enumerate}
\item\label{item:uparrowRelatesToNoUparrow:LemmSigmaOne}
Given a worm $A$, we have that $o(1\uparrow A)=-1 + \omega^{o(A)}$;
\item\label{item:computationSigmaOne:LemmSigmaOne}
For each ordinal $\xi$ we have $\sigma^1 \xi = -1 + \omega^\xi$.
\end{enumerate}
\end{lemma}

\proof
We first prove \ref{item:uparrowRelatesToNoUparrow:LemmSigmaOne} by induction on $<_0$. For the base case, $A=1{\uparrow} A=\top$, we verify that $o(1{\uparrow} A) = 0=-1+\omega^0 = -1+\omega^{o(A)}$.

If $A$ is a limit worm, i.e.\ $h(A)\neq \top$, the claim follows from Lemmas \ref{theorem:oIsContinuous} and \ref{theorem:UparrowIsContinuous} 
since
\[
-1+ \omega^{o(A)}=\sup_{B<A}(-1 + \omega^{o(B)}) \stackrel{\rm IH}= \sup_{B<A}o(1{\uparrow}B)= o(1{\uparrow} \sup_{B<A}B) = o(1{\uparrow A}).
\]

If $A=0B$, then by Lemma \ref{theorem:FundamentalSequencesForWormsStartingWith1} we have that $1{\uparrow}A=\sup_{n<\omega}B_n$, where $B_0:=1{\uparrow}B$ and $B_{n+1}:=(1{\uparrow} B)0B_n$. Therefore $o(1\uparrow A)=\sup_{n<\omega}o(B_n)$. 

By an easy subsidiary induction on $n$ we now see that $o(B_{n}) = -1 + \omega^{o(B)}\cdot (n+1)$. For $n=0$ this is just the induction hypothesis of the main induction. For $n+1$ we apply Lemma \ref{theorem:basicPropertyOrder:successors} to obtain 
\[
o(B_{n+1})= o((1{\uparrow}B)0B_n) = o(B_n) + 1 + o((1{\uparrow}B)) 
\]
By the subsidiary induction we have that $ o(B_n) = -1 + \omega^{o(B)}\cdot (n+1)$ and by the main induction we have that $o((1{\uparrow}B)) = -1 + \omega^{o(B)}$ so that 
\[
o(B_n) + 1 + o((1{\uparrow}B)) = -1 + \omega^{o(B)}\cdot (n+1) + 1 + -1 + \omega^{o(B)} = -1 + \omega^{o(B)}\cdot (n+2)
\]
as was to be shown. We now conclude the argument by observing that for $A=0B$ we have
\[
o(1{\uparrow}A) = \sup_{n<\omega} o(1{\uparrow}B_n) =
\sup_{n<\omega} -1+ \omega^{o(B)}\cdot (n+1) = -1 + \omega^{o(B)+1}= -1 + \omega^{o(0B)}
\]
whence $o(1{\uparrow}A) = -1+ \omega^{o(A)}$, as claimed.\\
\medskip

To see Item \ref{item:computationSigmaOne:LemmSigmaOne}, choose an ordinal $\xi$ and a worm $A$ such that $o(A)=\xi$. By Theorem \ref{theorem:OrderTypeCalculus}.\ref{exo}, we have that $\sigma^1\x=\sigma^1 o(A) = o(1\uparrow A)$; but by the previous item this is equal to $-1+\omega^{o(A)}=-1+\omega^\xi$, as desired.
\endproof

Note that $-1+ \omega^\xi = \omega^\xi$ whenever $\xi\neq 0$.
At this point we may give a convenient breakdown of $o(A)$ in terms of its head, rest and body.

\begin{corollary}\label{CorAltCalc}
If $A$ is any worm, then
\[o(A)=o(r(A))+o(h(A))=o(\body A)+\omega^{o(1\downarrow h(A))}.\]
\end{corollary}

\proof
Immediate from Theorem \ref{theorem:OrderTypeCalculus} and Lemma \ref{LemmSigmaOne}.
\endproof

With a lemma similar to Lemma \ref{LemmSigmaOne} we can now characterize $\sigma^2$ since $2{\uparrow} = 1{\uparrow}\circ 1{\uparrow}$ and we just need to iterate Lemma \ref{LemmSigmaOne}. The more general fact that $(\alpha + \beta){\uparrow}= \alpha{\uparrow} \circ \beta{\uparrow}$ is reflected in the following lemma.

\begin{lemma}\label{theorem:eSupDefinesWeakHyperation}\ \\
\vspace{-.4 cm}
\begin{enumerate}
\item
$\sigma^0 = {\sf id}$;

\item
$\sigma^1 = e$ where $e(\xi) = -1 + \omega^\xi$;
\item
$\sigma^{\alpha + \beta} = \sigma^{\alpha}\circ \sigma^{\beta}$.
\end{enumerate}
\end{lemma}

\begin{proof}
The first item is Lemma \ref{theorem:sigmaZeroIsTheIdentity} and the second item is Lemma \ref{LemmSigmaOne}.

For the last item we see that
\begin{align*}
\sigma^{\alpha}\circ \sigma^{\beta} &= \oAlg\circ\alpha{\upAlg} \circ \oAlg^{-1}\circ \oAlg\circ\beta{\upAlg} \circ \oAlg^{-1}
= \oAlg\circ\alpha{\upAlg} \circ\beta{\Uparrow} \circ \oAlg^{-1}\\
&= \oAlg\circ(\alpha+\beta){\upAlg} \circ \oAlg^{-1}
= \sigma^{\alpha+\beta}.
\end{align*}
\end{proof}


Clearly this lemma, together with Theorem \ref{theorem:OrderTypeCalculus} completely determines the order-types for worms that only use natural numbers as ordinals.

\begin{example}\label{example:OrderComputationOfOmegaWorm}
$o(2103) = o(3) + 1+ o(21) = \sigma^3 o(0) + 1 + o(21) =  \sigma^3 (1) + 1 + o(21)= \sigma^1\circ \sigma^1\circ \sigma^1 (1) + 1 +o(21) = \omega^{\omega^\omega} +1+ o(21) = \omega^{\omega^\omega} +1+ \sigma^1 o(10)= \omega^{\omega^\omega} + 1+\sigma^1 o(10\top) =  \omega^{\omega^\omega} +1+ \sigma^1 (o(\top) + 1 + o(1))=  \omega^{\omega^\omega} +1+ \sigma^1\omega=  \omega^{\omega^\omega} + 1+{\omega^\omega}=\omega^{\omega^\omega} + {\omega^\omega}$.
\end{example}

\subsection{A recursive calculus for the  worm enumerators}

It is evident that Lemma \ref{theorem:eSupDefinesWeakHyperation} says nothing about the behavior of $\sigma^\alpha$ for additively indecomposable $\alpha$. To deal with those ordinals we have the following lemmas. 

%

\begin{lemma}\label{LemmLambdaLim}
If $\lambda$ is infinite and additively indecomposable then
\[
\lambda\uparrow (0A)=\sup_{\eta<\lambda}\eta(\lambda\uparrow A).
\]
\end{lemma}

\proof
It is evident that $\sup_{\eta<\lambda}\eta(\lambda\uparrow A)\leq \lambda (\lambda \uparrow A) = \lambda \uparrow (0A)$ so we shall show the other inequality by proving that for each $B<\lambda\uparrow (0A)$ there is some $\eta < \lambda$ so that $B\leq \eta (\lambda \uparrow A)$.
We distinguish two cases.

First assume that $B=\lambda\uparrow B'$. Then, $B'<0A$ and thus $B'\leq A$ so that $0(\lambda\uparrow B')\leq 0(\lambda \uparrow A)<\lambda\uparrow(0A)$.

Otherwise, there are $\gamma<\lambda$, a worm $B_1 \in \Worms_1$ and $B_0 \in \Worms$ such that $B=\gamma\uparrow(B_10B_0)$. By induction on length, there are $\eta_0,\eta_1<\lambda$ such that $B_i<\eta_i(\lambda\uparrow A)$. Letting $\eta'=1+\max\{\eta_0,\eta_1\}$ (so that $\eta'>0$) we see that $B_1<\eta'(\lambda \uparrow A)$ whence $B_1<_1\eta'(\lambda \uparrow A)$. Thus,  $\eta'(\lambda\uparrow A)\vdash 1B_1 \wedge 0B_0\vdash 1B_10B_0\vdash 0B_10B_0$ whence $B_10B_0<\eta'(\lambda\uparrow A)$. Since $\lambda$ is additively indecomposable we see that $\gamma + \lambda = \lambda$ and $\eta=\gamma+\eta'<\lambda$, while $B=\gamma\uparrow(B_10B_0)<\eta(\lambda\uparrow A)$, as needed.
\endproof

For any ordinal $\lambda$ we have that $\sigma^\lambda(0) =o (\top) = 0$. Moreover, since $\sigma^\lambda$ is continuous, we can compute $\sigma^\lambda$ on limit ordinals if we have computed the values for all smaller ordinals. Thus, we only need to study the behavior of $\sigma^\lambda$ on successor ordinals for which we have the next lemma.

\begin{lemma}\label{theorem:successorArgumentsAtIndecomposableStages}
Let $\lambda$ be an additively indecomposable limit ordinal. We have that 
\[
\sigma^{\lambda} (\beta +1) = \sup_{\eta < \lambda} \sigma^{\eta} (\sigma^{\lambda}(\beta) +1).
\]
\end{lemma}

\begin{proof}
Pick $B'$ so that $o(0B') = \beta + 1$ and let $B:= \lambda\uparrow B'$ so that by Corollary \ref{theorem:orderAndArrows} we obtain
\begin{equation}\label{equation:one}
o(B) = o(\lambda \uparrow B') = \sigma^\lambda o(B') = \sigma^\lambda \beta.
\end{equation}

Moreover, as $\lambda$ is additively indecomposable, we see that $-\eta+\lambda = \lambda$ for any $\eta<\lambda$. In particular we get that 
\begin{equation}\label{equation:two}
\eta\downarrow B \ =\  B \ \ \ \ \ \ \mbox{ for any $\eta<\lambda$}.
\end{equation}
By Lemma \ref{LemmLambdaLim}, $\lambda B=\sup_{\eta<\lambda}\eta B,$ so that
\begin{equation}\label{equation:three}
o(\lambda B) = \sup_{\eta<\lambda}\  o(\eta B).
\end{equation}
Recall that $o(0B') = \beta +1$ so that we can reason
\[
\begin{array}{llll}
\sigma^{\lambda} (\beta +1) & = \sigma^{\lambda} (o(0B')) & \\
 & = o(\lambda B) & \mbox{by Corollary \ref{theorem:orderAndArrows}}\\
 & = \sup_{\eta<\lambda} \ o(\eta B) &  \mbox{by (\ref{equation:three})}\\
 & = \sup_{\eta<\lambda}\  \sigma^{\eta} o(0 (\eta\downarrow B)) & \mbox{by Corollary \ref{theorem:orderAndArrows}}\\
 & = \sup_{\eta<\lambda} \ \sigma^{\eta}(o(\eta\downarrow B)+1) & \mbox{by Lemma \ref{LemmOrdBasic}}\\
 & = \sup_{\eta<\lambda}\ \sigma^{\eta}(o(B)+1) & \mbox{by (\ref{equation:two})}\\
 & = \sup_{\eta<\lambda}\ \sigma^{\eta}(\sigma^\lambda(\beta)+1) & \mbox{by \eqref{equation:one}}.\\
\end{array}
\]
\end{proof}

Now that we have proved this lemma we finally have fully determined all functions $\sigma^\alpha$.

\begin{theorem}\label{theorem:RecursiveHyperexponential}
For ordinals $\alpha$ and $\beta$, the values $\sigma^\alpha(\beta)$ are determined by the following recursion.
\begin{enumerate}
\item
$\sigma^\alpha 0 =0$ for all $\alpha \in \ord$;
\item
$\sigma^1 = e$ with $e(\xi) = -1 + \omega^\xi$;
\item
$\sigma^{\alpha + \beta} = \sigma^{\alpha} \sigma^{\beta}$;
\item
$\sigma^\alpha (\lambda) = \sup_{\beta<\lambda} \sigma^\alpha (\beta)$ for limit ordinals $\lambda$;

\item
$\sigma^{\lambda} (\beta +1) = \sup_{\eta < \lambda}\  \sigma^{\eta} (\sigma^{\lambda}(\beta) +1)$ for $\lambda$ an additively indecomposable limit ordinal.

\end{enumerate}
\end{theorem}

It is clear that this theorem embodies a full calculus. Let us see a simple example.

\begin{example}
$\sigma^\omega 1 = \varepsilon_0$ so that $o(\la \omega \ra \top) = \varepsilon_0$ with $\varepsilon_0 := \sup \{ \omega, {\omega^\omega}, \omega^{\omega^\omega}, \ldots\}$.
\end{example}

\begin{proof}
By definition, $\sigma^1(1) = e1 = \omega$. Consequently, $\sigma^2(1)=\sigma^1 \sigma^1(1) = \omega^\omega$ and likewise $\sigma^3(1)=\omega^{\omega^\omega}$, etc. Thus, $\sigma^\omega 1 = \sup_{\eta < \omega} \sigma^\eta (\sigma^\omega (0) +1)= \sup_{\eta < \omega} \sigma^\eta (1) = \varepsilon_0$.
\end{proof}

It should not come as a surprise that $\sigma^\omega(1)$ is a fixpoint of $e$ and something more general holds. If $\lambda$ is additively indecomposable, then $\eta + \lambda = \lambda$ for all $\eta < \lambda$. Thus 
\[
\sigma^\lambda (\xi) = \sigma^{\eta +\lambda} (\xi) = \sigma^{\eta} \sigma^{\lambda} (\xi)
\]
so that for each $\xi$ we see that $\sigma^\lambda (\xi)$ is a fixpoint of $\sigma^\eta$ for each $\eta<\lambda$; this is very similar to what happens with the Veblen hierarchy. In Theorem \ref{theorem:hyperationsVersusVeblenProgressions} the relation between the worm enumerators and the one-placed Veblen functions is made precise.




%








\section{Hyperations and Cohyperations}\label{section:OrdinalArithmetic}


%

A main theme of this paper is how to compute given a worm $A$ and ordinal $\xi$ its $\xi$-consistency order-type $o_\xi(A)$. In Corollary 
\ref{corollary:reductionXiOrdertypeToPlainOrdertype} 
we have reduced the $o_\xi$ order-types to the plain $o$ order-type. Subsequently, in Theorem \ref{theorem:OrderTypeCalculus} we provided a calculus for $o$ in terms of the so-called worm-enumerators $\sigma^\alpha$. Finally, in Theorem \ref{theorem:RecursiveHyperexponential} we worked out a recursive calculus for computing the worm-enumerators thereby completing all ingredients needed to compute any order-type $o_\xi(A)$.

In the final section of this paper we wish to characterize what, given a worm $A$, the sequence $\la o_\xi(A) \ra_{\xi \in {\sf On}}$ can look like. It shall turn out that to give a smooth characterization of these sequences, we need certain well-behaved left-inverses of our worm enumerator functions. These inverses can be computed within the general framework of what the authors call \emph{hyperations} and \emph{co-hyperations}.

Hyperations are a kind of transfinite iteration of certain ordinal functions and were introduced and studied in full generality by the authors in \cite{FernandezJoosten:2012:Hyperations}. In this section we shall briefly state --without proof-- the main properties of hyperations and the related \emph{cohyperations} that we need in the remainder of this paper and refer to \cite{FernandezJoosten:2012:Hyperations} for further background. Moreover, we shall prove that the worm-enumerators $\sigma^\alpha$ are the hyperation of a special form of ordinal exponentiation. For definitions and basic properties of ordinals, we refer the reader to \cite{Jech:2002:SetTheory,Pohlers:2009:PTBook}. 

\subsection{Hyperations}\label{section:hyperations}

As mentioned before, {\em hyperation} is a form of transfinite iteration of normal functions. It is based on the additivity of finite iterations, that is $f^{m+n} = f^m\circ f^n$ generalizing this to the transfinite setting. Let us first recall the definition of a normal function.

We call a function on the ordinals $f$ \emph{increasing} if $\alpha < \beta$ implies $f(\alpha) < f(\beta)$. An ordinal function is called \emph{continuous} if $\sup_{\zeta<\xi}f(\zeta) = f(\xi)$ for all limit ordinals $\xi$. Functions which are both increasing and continuous are called \emph{normal}.

\begin{defi}[Weak hyperation]\label{definition:WeakHyperation}
A {\em weak hyperation} of a normal funcion $f$ is a family of normal functions $\langle g^{\x}\rangle_{\xi\in\mathsf{On}}$ such that
\begin{enumerate}
\item $g^0{\x}={\x}$ for all ${\x}$,
\item $g^1=f$,
\item $g^{{\x}+\zeta}=g^{\x} g^\zeta$.
\end{enumerate}
\end{defi}

\emph{Par abuse de langage} we will often write just $g^{\x}$ instead of $\langle g^{\x}\rangle_{\xi\in\mathsf{On}}$. In Lemma \ref{theorem:eSupDefinesWeakHyperation} we have proven that the family of worm enumerators $\sigma^\xi$ is a weak hyperation of the function $e$ defined as $\xi \mapsto -1 + \omega^\xi$.

Weak hyperations are not unique. However, if we impose a minimality condition, we can prove that there is a unique minimal hyperation.

\begin{defi}[Hyperation]
A weak hyperation $g^{\x}$ of $f$ is {\em minimal} if it has the property that, whenever $h^{\x}$ is a weak hyperation of $f$ and ${\x},\zeta$ are ordinals, then $g^{\x}\zeta\leq h^{\x}\zeta$.

If $f$ has a (unique) minimal weak hyperation, we call it {\em the hyperation} of $f$ and denote it $f^{\x}$.
\end{defi}

We shall now prove that the worm enumerators $\sigma^\xi$ are the hyperation of the function $e$.

\begin{theorem}\label{theorem:wormenumeratorsAreMinimalHyperation} 
$\sigma^\alpha$ is the hyperation of the function $e \ : \ \xi \mapsto -1 + \omega^\xi$.
\end{theorem}

\begin{proof}
The properties 1.--3. of Lemma \ref{theorem:eSupDefinesWeakHyperation} express that the $\sigma^\alpha$ are a weak hyperation of $e$. To see that it is the unique hyperation we only need to check for minimality.

So, suppose that $\{f^\alpha\}_{\alpha \in {\sf Ord}}$ is a collection of normal functions such that 1.--3. holds. By induction on $\alpha$ we shall see that $\sigma^\alpha(\beta) \leq f^{\alpha} (\beta)$. 

For $\alpha =0$ and $\alpha =1$ this is obvious and for additively decomposable ordinals we see that $\sigma^{\alpha + \beta} = \sigma^\alpha \sigma^\beta \leq_{\sf IH} f^\alpha f^\beta = f^{\alpha + \beta}$.

So, let $\alpha$ be an indecomposable limit ordinal. We proceed by an auxiliary induction on $\beta$ to show that $\sigma^\alpha(\beta) \leq f^{\alpha} (\beta)$ which clearly holds for $\beta =0$. As both $f^\alpha$ and $\sigma^\alpha$ are continuous, we only need to consider successor ordinals in which case, by Lemma \ref{theorem:successorArgumentsAtIndecomposableStages} we see that
\begin{equation}\label{SigmaLess}
\sigma^{\alpha} (\beta +1) = \sup_{\alpha' < \alpha}\  \sigma^{\alpha'} (\sigma^{\alpha}(\beta) +1) \leq_{\sf IH}  \sup_{\alpha' < \alpha} \ f^{\alpha'} (f^{\alpha}(\beta) +1). \end{equation}
As for $\alpha' < \alpha$ we have $\alpha'+\alpha = \alpha$, by Property 3, we see that $f^\alpha (\beta +1) = \sup_{\alpha'<\alpha}\  f^{\alpha'}f^\alpha (\beta +1)$. But, as $f^\alpha$ is monotone we also see that $f^\alpha (\beta +1) \geq f^\alpha (\beta)+1$ whence by monotonicity of all of the $f^{\alpha'}$ we see that 
\[
f^\alpha (\beta +1) = \sup_{\alpha'<\alpha}\  f^{\alpha'}f^\alpha (\beta +1) \geq \sup_{\alpha'<\alpha}\  f^{\alpha'}(f^\alpha (\beta) +1).
\]
We combine this with \eqref{SigmaLess} to conclude that $\sigma^{\alpha} (\beta +1) \leq f^\alpha (\beta +1)$.
\end{proof}


Since we have proven that $e^\xi = \sigma^\xi$, from now on we shall only use the notation based on $e$. We call the family $e^\xi$ \emph{hyperexponentials}.

Hyperations in general allow for an explicit recursive definition very much in the style of Theorem \ref{theorem:RecursiveHyperexponential}. Moreover, there turns out to be a close connection between hyperations and Veblen progressions: 

\begin{definition}
For $f$ a normal function, the Veblen progression based on $f$ is denoted by $\la f_\xi\ra_{\xi \in {\sf On}}$ and defined by $f_0 := f$ and for $\xi>0$, $f_\xi$ is the normal function that enumerates in increasing order the ordinals which are simultaneous fixpoints for all the $f_\zeta$ for $\zeta<\xi$.
\end{definition}

The Veblen progression based on $\varphi (\xi) := \omega^\xi$ are the well-known one-place Veblen functions $\varphi_\xi$. Note that $\varphi (\xi) = e(\xi)$ for $\xi\neq 0$. In \cite{FernandezJoosten:2012:Hyperations} it is proven that hyperations can be seen as a natural refinement of Veblen progressions.

\begin{theorem}\label{theorem:hyperationsVersusVeblenProgressions}
For $f$ a normal function we have that $f^{\omega^\xi} = f_\xi$.
\end{theorem}

It should be clear that this theorem provides an easy link between the hyperexponentials and the earlier known Veblen functions. In particular, if we write some ordinal $\xi$ in its unique Cantor Normal Form with base $\omega$ so that for some $n\geq 0$ we have $\xi = \omega^{\xi_1} + \ldots +\omega^{\xi_n}$ and $\xi_i \geq \xi_{i+1}$ for all $i<n$, then for $\alpha >0$ and $\xi \notin {\sf Lim}$ (for $\xi \in {\sf Lim}$ we should replace the $\alpha$ on the right side by $-1+\alpha$):
 
\[
e^\xi (\alpha) = \varphi_{\xi_1}\circ \ldots \circ \varphi_{\xi_n}(\alpha).
\]


\subsection{Cohyperations}

We shall see that in order to relate the different $\xi$-consistency order-types $o_\xi$ to each other we shall need left inverses to hyperexponentials. Hyperations are injective and hence invertible on the left; however, a left-inverse of a hyperation is typically not a hyperation, but a different form of transfinite iteration we call {\em cohyperation}. 

Instead of transfinitely iterating normal functions we shall consider \emph{initial functions}. We will say a function $f$ is {\em initial} if, whenever $I$ is an initial segment (i.e., of the form $[0,\beta)$ for some $\beta$), then $f(I)$ is an initial segment. It is easy to see that $f\xi \leq \xi$ for initial functions $f$.


\begin{defi}[Cohyperation]
A {\em weak cohyperation} of an initial function $f$ is a family of initial functions $\langle g^{\x}\rangle_{\xi\in\mathsf{On}}$ such that
\begin{enumerate}
\item $g^0{\x}={\x}$ for all ${\x}$,
\item $g^1=f$,
\item $g^{{\x}+\zeta}=g^\zeta g^{\x}$.
\end{enumerate}

If $g$ is maximal in the sense that $g^{\x}\zeta\geq h^{\x}\zeta$ for every weak cohyperation $h$ of $f$ and all ordinals $\xi,\zeta$, we say $g$ is {\em the cohyperation} of $f$ and write $f^{\x}=g^{\x}$.
\end{defi}

Both hyperations and cohyperations are denoted using the superscript; however, this does not lead to a clash in notation as the only function that is both normal and initial is the identity. In \cite{FernandezJoosten:2012:Hyperations}, a general recursive scheme to compute actual cohyperations is given very much in the spirit of Theorem \ref{theorem:RecursiveHyperexponential}. 

\medskip

Let $f$ be a normal function. Then, $g$ is a {\em left adjoint} for $f$ if, for all ordinals $\alpha,\beta$,
\begin{enumerate}
\item if $\alpha= f(\beta)$, then $g(\alpha)=\beta$ and
\item if $\alpha< f(\beta)$, then $g(\alpha)<\beta$.
\end{enumerate}

Left-adjoints are natural left-inverses and cohyperating them yields left-adjoints to the corresponding hyperations in a uniform way:

\begin{theorem}\label{cancel}
Given a normal function $f$ with left adjoint $g$ and ordinals $\xi\leq\zeta$ and $\alpha$, $g^\xi f^\zeta=f^{-\xi+\zeta}$ and $g^\zeta f^\x=g^{-\xi+\zeta}$.
\end{theorem}

We shall need left-adjoints to our hyperexponentials. In order to formulate them, let us first recall some basic properties of the ordinals.

\begin{lemma}
Given any ordinal $\eta>0$, there exist ordinals $\alpha,\beta$, where $\beta$ is uniquely determined, such that $\eta = \alpha + \omega^{\beta}.$ We will denote this unique $\beta$ by $\le \eta$ and define $\le 0=0$.

\end{lemma}

The following theorem is proven in \cite{FernandezJoosten:2012:Hyperations}.

\begin{theorem}\label{logexp}
The function $\le$ is a left adjoint to $\ex$, and thus $\le^\xi$ is left adjoint to $\ex^\x$ for all $\x$.
\end{theorem}

We will refer to the functions $\le^\xi$ as {\em hyperlogarithms}.

\subsection{Exact sequences}

A nice feature of cohyperations is that, in a sense, they need only be defined locally. To make this precise, we introduce the notion of an {\em exact sequence}.

\begin{defi}
Let $g^{\x}$ be a cohyperation, and $f:\Lambda\to\Theta$ be an ordinal function.

Then, we say $f$ is {\em $g$-exact} if, given ordinals ${\x},\zeta$ with $\x + \zeta < \Lambda$, $f({\x}+\zeta)=g^\zeta f({\x})$.
\end{defi}

A $g$-exact function $f$ describes the values of $g^{\x} f(0)$. However, for $f$ to be $g$-exact, we need only check a fairly weak condition:

\begin{lemma}\label{tfae}
The following are equivalent:
\begin{enumerate}
\item $f$ is $g$-exact
\item for every ordinal ${\x},$ $f({\x})=g^{\x} f(0)$
\item for every ordinal $\zeta>0$ there is ${\x}<\zeta$ such that $f(\zeta)=g^{-{\x}+\zeta}f({\x})$.
\end{enumerate}
\end{lemma}

\begin{example}\label{example:anExactSequence}
By $e^{\alpha}\beta \cdot \gamma$ we denote $(e^{\alpha}\beta) \cdot \gamma$. Then, the following sequence whose initial sub-sequence of non-zero elements is of length $\omega \cdot 2 + 2$ is $\le$-exact: 
\[
\la e^{\omega^2}1 + e^\omega (e^{\omega+1}1 \cdot 2), e^\omega (e^{\omega+1}1 \cdot 2),  \ldots, e^{\omega+1}1 \cdot 2,e^{\omega+1}1, \ldots , \omega, 1, 0 \ldots \ra
\]
\end{example}

\begin{proof}
By Lemma \ref{tfae} and Theorem \ref{logexp}. 
\end{proof}

In the light of Theorem \ref{theorem:hyperationsVersusVeblenProgressions} we can reformulate this example in terms of the better-known Veblen functions. Thus, using the usual convention that $\varphi_1(\alpha)$ is denoted by $\varepsilon_\alpha$ we can rewrite our example as:
\[
\la \varphi_2(1) + \varepsilon_{\varepsilon_\omega + \varepsilon_\omega}, \varepsilon_{\varepsilon_\omega + \varepsilon_\omega},  \ldots, \varepsilon_\omega + \varepsilon_\omega,  \varepsilon_\omega, \ldots , \omega, 1, 0 \ldots \ra .
\]


\section{Consistency Sequences}\label{section:consistencySequences}


Given a worm $A$, we define its \emph{consistency sequence} to be the sequence $\la o_\xi(A) \ra_{\xi \in {\sf On}}$.
In this section we give a full characterization of consistency sequences. That is, we will determine which sequences $\<\alpha_\xi\>_{\xi\in\sf On}$ are of the form $\vec{\formerOmega}(A)$ for some worm $A$. Moreover, given ${\formerOmega}_0(A)$, we will compute ${\formerOmega}_\xi(A)$ for all $\xi>0$, even when $A$ itself is not explicitly given.

It is intuitively clear that for constant $A$, the function $o_\xi (A)$ is weakly decreasing in $\xi$ as is expressed in the following lemma.
 
\begin{lemma}\label{theorem:WeaklyDecreasingSequences}
For ${\x}<\zeta$ we have that ${\formerOmega}_{\x}({A})\geq {\formerOmega}_\zeta({A})$.
\end{lemma}

\proof
By induction on ${\formerOmega}_\x (A)$ we see that
\[
{\formerOmega}_\x (A) := \sup_{B <_\x A} \big({\formerOmega}_\x (B) +1 \big) \geq_{\sf IH}  \sup_{B <_\x A} \big({\formerOmega}_\zeta (B) +1 \big)\\
 \geq \sup_{B <_\zeta A}\big({\formerOmega}_\zeta (B) +1 \big)=  {\formerOmega}_\zeta (A).
\]
Note that we have the last inequality since for ${\x}<\zeta$ we have  $ B<_\zeta {A} $ implies $ B<_\xi {A}$. 
\qed

We will use the notation $\vec{{\formerOmega}}({A})$ for the sequence $\la{\formerOmega}_\xi(A)\ra_{\xi\in\ord}$; that is,
\[
\vec{{\formerOmega}}({A})\ \ :=\ \ \la {\formerOmega}_0({A}), \ {\formerOmega}_1({A}),\ \ldots , {\formerOmega}_{\omega}({A}),\ {\formerOmega}_{\omega+1}({A}),\ldots \ra \ .
\]
By Lemma \ref{theorem:WeaklyDecreasingSequences}, $\vec{{\formerOmega}}({A})$ is a weakly decreasing sequence of ordinals. In particular, we have that $\{ {\formerOmega}_\xi(A) \mid \xi \in \ord \}$ is a finite set for any worm $A$. Moreover, we see that any consistency sequence eventually hits zero.

\begin{corollary}\label{theorem:MaximalCoordinateOfSequences}
Given a worm $A\neq \top$, we may write $A=\xi B$ for a unique ordinal $\xi$ and worm $B$. Then, given an arbitrary ordinal $\zeta$, we have that $o_\zeta(A)=0$ if and only if $\zeta>\xi$.
\end{corollary}

\proof
If $A=\xi B$ then clearly $h_\xi(A)\not=\top$, so that by Lemma \ref{theorem:OmegaReducesToO},
${\formerOmega}_{\xi}({A})=o_\xi h_\xi(A)\neq 0$.
On the other hand, for $\zeta> \xi$, $h_{\zeta}({A}) = \top$ whence ${\formerOmega}_{\zeta}({A}) =0$.
\qed

\PreprintVersusPaper{
\subsection{Motivation for consistency sequences}

It is clear that within the framework of this paper it is a very natural question to study what the consistency sequences look like. Moreover, it turns out that a characterization of consistency sequences sheds light on various other problems. Let us first discuss a relation between consistency sequences and Kripke semantics for the closed fragment of \glp.

Each worm $A$ is known to be consistent with \glp, hence should be satisfied in an exact 
model for its closed fragment, if it exists; that is, a model on which only the theorems of $\glp^0$ are valid. Kripke models give us a transparent and convenient interpretation
for many modal logics. A {\em Kripke frame} is a structure
$\mathfrak F=\<W,\<R_\lambda\>_{\lambda<\Lambda}\>$, where $W$ is a
set and $\<R_\lambda\>_{\lambda<\Lambda}$ a family of binary
relations on $W$. Let $\mathcal L^0_\Lambda$ be the closed language of $\glp_\Lambda$. The {\em closed valuation} on $\mathfrak F$ is the unique function
$\lb\cdot\rb:\mathcal L^0_\Lambda\to \mathcal P(W)$ such that
\[
\begin{array}{lcllcl}
\lb\top\rb&=&W\ &\qquad
\lb\neg\phi\rb&=&W\setminus\lb\phi\rb\\
\lb\phi\wedge\psi\rb&=&{}\lb\phi\rb\cap\lb\psi\rb\ &\qquad
\lb\<\lambda\>\phi\rb&=&R^{-1}_\lambda\lb\phi\rb.
\end{array}
\]

One often writes $\<\mathfrak
F,\lb \cdot \rb\>,w\Vdash\psi$ instead of $w\in\lb \psi\rb$ or even just $w\Vdash\psi$ if the context allows it. As usual, $\phi$ is
{\em satisfied} on $\<\mathfrak
F,\lb \cdot \rb\>$ if $\lb\phi\rb\not=\varnothing$,
and {\em valid} on $\<\mathfrak
F,\lb \cdot \rb\>$ if $\lb\phi\rb=W$. The latter case shall be denoted by $\<\mathfrak F,\lb \cdot \rb\> \models \psi$.

As $[\xi]$ satisfies L\"ob's axiom, we know that each $R^{-1}_\xi$ is transitive and well-founded. 
Consequently, we can assign to each world $w\in W$ a sequence of ordinals 
\[
\vec{w} \ := \ ({w}_0, {w}_1, \ldots {w}_{\omega}, {w}_{\omega+1}\ldots),
\] 
where ${w}_\zeta$ corresponds to the $R^{-1}_\zeta$-order-type of $w$. If $\langle\mathfrak F,w\rangle\Vdash A$ for some worm $A$, then necessarily $w_\xi \geq {\formerOmega}_\xi(A)$ for each $\xi$. A systematic study of $\vec{{\formerOmega}}(A)$ will thus also reveal information about models for $\glp^0$.

Ignatiev presented an exact model for $\glp^0_\omega$ in \cite{Ignatiev:1993:StrongProvabilityPredicates} which has been studied in various other papers too \cite{Joosten:2004:InterpretabilityFormalized, BeklemishevJoostenVervoort:2005:FinitaryTreatmentGLP,Icard:2008:MastersThesis}. The worlds in this model are finite sequences of ordinals $\la\xi_i\ra_{i\in \omega}$ where $\xi_{n+1}\leq \le \xi_n$. In Theorem \ref{theorem:SuccessorRelations} below, we will explain \emph{why} this condition on the successor coordinates in these sequences needs to be imposed.

In \cite{FernandezJoosten:2012:ModelsOfGLP} the authors define an exact model for $\glp^0_\Lambda$ for any ordinal $\Lambda$. The worlds in those models closely reflect the consistency sequences as defined here.
Accordingly, the models have the property that if $\langle \mathfrak{F},w\rangle\Vdash A$, then  $w_\xi \geq {\formerOmega}_\xi(A)$ for each $\xi$. Thus, giving an exhaustive characterization of consistency sequences is fundamental in understanding semantics of $\glp^0$.\\
\medskip

Moreover, for consistency sequences of length $\omega$ there turns out to be a natural proof-theoretic interpretation. Refraining from technical details, let us denote for some theory $T$ by $T^n_\alpha$ the $\alpha$th Turing progression based on the $n$-consistency notion, that is $T^n_0 := T$, and $T^n_{\beta}:= \bigcup_{\alpha <\beta}\big(T^n_\alpha + n\mbox{-}{\sf Cons}(T^n_\alpha) \big)$, where $n\mbox{-}{\sf Cons}$ denotes a formalization of `is consistent together with all true $\Pi_n^0$ sentences'. 

Let \ea denote elementary arithmetic. We can define for a theory $U$ its $\Pi^0_{n+1}$ proof theoretical ordinal $|U|_{\Pi^0_{n+1}}: = \sup \{  \alpha \mid U \vdash \ea^n_\alpha \}$. Since worms $A$ are iterated consistency sequences, they naturally can be given arithmetic semantics $A^*$. For worms $A$ in $\glp_\omega$ it has been shown (\cite{FernandezJoosten:2012:WellOrders2}) that the sequence of proof theoretic ordinals of the theory $\ea + A^*$ is exactly given by $\vec o (A)$.
}{

Consistency sequences are of interest of their own. Moreover, they shed light on (Kripke) semantics for the closed fragment of \glp. Also, they admit a proof-theoretic interpretation as explained in \cite{FernandezJoosten:2012:WellOrders2}.}


\subsection{A local characterization}


We shall first provide a local characterization of the consistency sequences in that we relate the values in the sequence to its neighbors. To this end, let us first compute ${\formerOmega}_{\x+1}(A)$ in terms of ${\formerOmega}_\x(A)$. Recall that $\le \alpha$ denotes the unique $\beta$ such that $\alpha = \alpha' + \omega^{\beta}$ for $\alpha >0$, while $\le 0 = 0$. The following lemma will be useful:

\begin{lemma}\label{lemma:successorLemma}
Given an ordinal ${\x}$ and a worm ${A}$, $o_{{\x}+1}h_{{\x}+1}({A}) = \le o_{\x} h_{\x}({A})$.
\end{lemma}

\proof
Let $B=\xi\downarrow h_\x (A)$, so that from Corollary \ref{CorAltCalc} we have
\[o_\x h_\x (A)=o (\x \downarrow h_\x(A))=o(B)= o\body B+\omega^{o(1\downarrow h(B))},\]
and thus $\le o_\x h_\x(A)=o(1\downarrow h(B))$. But $h(B)=\xi\downarrow h_{\xi+1}(A)$, so
\[o(1\downarrow h(B))=o((\xi+1)\downarrow h_{\xi+1}(A))=o_{\x+1}h_{\x+1}(A),\]
where the last equality is an instance of Corollary \ref{theorem:orderAndArrows}.
\qed

Now we are ready to describe the relation between successive coordinates of the $\vec{{\formerOmega}}({A})$ sequence.

\begin{theorem}\label{theorem:SuccessorRelations}
Given an ordinal $\x$ and a worm $A$, ${\formerOmega}_{{\x}+1}({A}) = \le{\formerOmega}_{\x}({A})$
\end{theorem}
\proof We have that ${\formerOmega}_{{\x}+1}({A}) = o_{{\x}+1}h_{{\x}+1}({A}) = \le o_{\x} h_{\x}({A})=\le{\formerOmega}_{\x}({A}),$ where the second equality uses Lemma \ref{lemma:successorLemma} and the others Lemma \ref{theorem:OmegaReducesToO}.
\qed

Theorem \ref{theorem:SuccessorRelations} tells us what the relation between successor coordinates of $\vec{{\formerOmega}}({A})$ is. We may also infer from it when successor coordinates are different; if ${\formerOmega}_{\x}({A})$ is a fixed point of $\zeta \mapsto -1+\omega^\zeta$ then ${\formerOmega}_{\x}({A})= {\formerOmega}_{{\x}+1}({A})$. 

Next we shall determine what happens at limit steps in the consistency sequences. Since we know that for a given worm $A$, the set $\{ o_\xi (A)\mid \xi \in {\sf On}\}$ is finite it is clear that for limit $\xi$, the value $o_\xi(A)$ can only be non-zero, if at some point before $\xi$, the sequence $\vec o (A)$ had stabilized. We shall now compute the relation between this stabilized value and the limit value.
Here, our functions $\ex^\x$ come back into play:

\begin{theorem}\label{theorem3.6}
Let $\y{\in}{\sf Lim}$ then, for $\theta$ large enough we have that
\[
{\formerOmega}_{\theta}({A})  =  \ex^{-\theta+\y} {\formerOmega}_\zeta({A}) =  e^{\omega^{\le \zeta}} {\formerOmega}_\zeta(A)=  e_{{\le \zeta}} ({\formerOmega}_\zeta(A)).
\]
\end{theorem}

\proof
That $e^{\omega^{\le \zeta}} {\formerOmega}_\zeta(A)= e_{{\le \zeta}} ({\formerOmega}_\zeta(A))$ is just Theorem \ref{theorem:hyperationsVersusVeblenProgressions} so we focus on the first equalities. Since $\zeta\in\sf Lim$, for $\xi$ large enough we have that $h_\zeta(A)=h_\xi(A)$, and more generally,
\begin{equation}\label{dagger}
h_\zeta(A)=h_\theta(A)
\end{equation}
whenever $\theta\in[\xi,\zeta]$. Moreover, writing $\zeta=\zeta'+\omega^{\le \zeta}$, we may without loss of generality choose $\xi\geq \zeta'$ so that $-\theta+\zeta=\omega^{\le\zeta}$ for all $\theta\in[\xi,\zeta).$ For the sake of abbreviating, let $\delta=-\theta+\y = \omega^{\le \zeta}$. By definition
\begin{equation}\label{equation:DefinitionOfDelta}
\theta + \delta = \y.
\end{equation}

Now we can prove our theorem starting with an application of Lemma \ref{theorem:OmegaReducesToO}: 
\[
\begin{array}{llll}
{\formerOmega}_{\theta}({A}) & = & o_{\theta}h_{\theta}({A}) &  \\
 & = & o_{\theta}h_\y({A})&  \mbox{By (\ref{dagger})}\\
 & = &o( \theta\downarrow h_\zeta({A}))&  \mbox{Corollary \ref{corollary:reductionXiOrdertypeToPlainOrdertype}}\\
 & = & o(\delta\uparrow ((\theta + \delta)\downarrow h_\zeta (A))) & \mbox{Lemma \ref{theorem:uparrowProperties}.\ref{inverse}}\\
 & = & o(\delta\uparrow (\zeta\downarrow h_\zeta (A)))&  \mbox{By \eqref{equation:DefinitionOfDelta}}\\
 & = & \ex^{\delta} o(\zeta{\downarrow}h_\zeta({A})) &  \mbox{Theorem \ref{theorem:OrderTypeCalculus}.\ref{exo}}\\
 & = & \ex^{\delta} {\formerOmega}_\zeta({A}) &\mbox{Corollary \ref{corollary:reductionXiOrdertypeToPlainOrdertype}}\\
\end{array}
\]
\qed

Theorem \ref{theorem:SuccessorRelations} and Theorem \ref{theorem3.6} have been presented in a somewhat different guise in \cite{FernandezJoosten:2012:TuringProgressions}.
Note that, indeed, these theorems provide a local characterization of $\vec o(A)$ given any worm $A$: start by computing $o_0(A)$ and compute how $o_\xi (A)$ changes for increasing values of $A$. Moreover, in \cite{FernandezJoosten:2012:TuringProgressions} a characterization is given for those values $\xi$ where $o_\xi (A)$ changes value.

\subsection{From local to global}\label{section:FromLocalToGlobal}

The computations we have presented give the value of ${\formerOmega}_\xi(A)$ from ${\formerOmega}_\zeta(A)$ for $\zeta<\xi$ {\em provided $\zeta$ is large enough.} As such, we have only characterized them locally. In the next subsection we will give a global characterization of $\vec{\formerOmega}(A)$, so that all values may be computed directly from ${\formerOmega}_0(A)$.

In our local characterization we have distinguished two cases: successor coordinates and limit coordinates. 
It turns out that one can conceive both successor and limit steps as one of the same kind. For the successor case when $\zeta = \xi + 1$ we saw that

\begin{equation}\label{equation:successorCharacterization}
{\formerOmega}_\zeta(A)= {\formerOmega}_{\xi + 1} (A) = \le {\formerOmega}_{\xi} (A) = \le^{-\xi + (\xi +1)} {\formerOmega}_{\xi} (A) = \le^{-\xi + \zeta}{\formerOmega}_\xi (A).
\end{equation}

For limit steps, we say ${\formerOmega}^\xi (A) = e^{-\xi + \zeta} {\formerOmega}_\zeta(A)$ for $\xi < \zeta$ large enough. By Lemma \ref{logexp}, $\le^\alpha$ is a left-inverse to $e^\alpha$ for all $\alpha$. Then our characterization for limit coordinates becomes
\begin{equation}\label{equation:limitCharacterization}
{\formerOmega}_{\zeta} (A) =  \le^{-\xi + \zeta} {\formerOmega}_{\xi} (A) \ \ \ \ \ \mbox{for $\xi<\zeta$ large enough.}
\end{equation}

Written in this way, we see that  \eqref{equation:successorCharacterization} and \eqref{equation:limitCharacterization} actually are the same.  Moreover, as we shall see, Lemma \ref{tfae} will allow us to drop the requirement ``for $\xi<\zeta$ large enough" giving rise to our desired global characterization. Let us unify the results obtained so far by describing the sequences $\vec{\formerOmega}({A})$ using hyperexponentials and hyperlogarithms.

\subsection{A global characterization}

\begin{theorem}\label{istheexact}
If $A$ is any worm, $\vec{\formerOmega}({A})$ is the unique $\le$-exact sequence with ${\formerOmega}_0({A})=o({A})$.
\end{theorem}

\proof
In view of Lemma \ref{tfae}, it suffices to show that, given any ordinal $\zeta$, there is ${\x}<\zeta$ such that ${\formerOmega}_\zeta({A})=\le^{-{\x}+\zeta}{\formerOmega}_{\x}({A})$.

If $\zeta$ is a successor ordinal, write $\zeta={\x}+1$. Then, by Theorem \ref{theorem:SuccessorRelations}, we have that ${\formerOmega}_\zeta({A})=\le{\formerOmega}_{\x}({A})$.

Meanwhile, if $\zeta$ is a limit ordinal, we know from Lemma \ref{theorem3.6} that, for ${\x}<\zeta$ large enough, ${\formerOmega}_{\x}({A}) = \ex^{-\x+\y}{\formerOmega}_\zeta({A})$. Applying $\le^{-\x+\zeta}$ on both sides and using Theorem \ref{cancel}, we see that $\le^{-\x+\zeta}{\formerOmega}_{\x}({A}) = {\formerOmega}_\zeta({A}).$ Thus we can use Lemma \ref{tfae} to see that $\vec{\formerOmega}(A)$ is $\le$-exact, so that, for all $\xi$, ${\formerOmega}_\xi(A)=\le^\xi{\formerOmega}_0(A)$, as claimed.
\qed

\begin{example}
For $A = \langle\omega\cdot 2 +1\rangle\langle\omega\rangle\langle\omega\cdot 2 +1\rangle\langle 0\rangle \langle\omega^2\rangle\top$ we have that 
\[
\vec{{\formerOmega}}(A)=  \la e^{\omega^2}1 + e^\omega (e^{\omega+1}1 \cdot 2), e^\omega (e^{\omega+1}1 \cdot 2),  \ldots, e^{\omega+1}1 \cdot 2,e^{\omega+1}1, \ldots , \omega, 1, 0 \ldots \ra.
\]

\end{example}

\PreprintVersusPaper{
\begin{proof}
In virtue of Theorem \ref{istheexact} above and Example \ref{example:anExactSequence} it suffices to see that $o(A) = e^{\omega^2}1 + e^\omega (e^{\omega+1}1 \cdot 2)$.
\begin{enumerate}
\item
$o(\omega+1) = o((\omega+1)\uparrow 0) = e^{\omega+1}o(0)=e^{\omega+1}1$;
\item
$o((\omega +1) 0 (\omega+1)) = \omega^{o(1\downarrow (\omega+1))}\cdot 2 = \omega^{e^{\omega+1}1}\cdot 2 = e^{\omega+1}1\cdot 2$;
\item
$o(\, (\omega\cdot 2 +1) \omega (\omega\cdot 2 +1)\, ) = o(\,\omega\uparrow ((\omega +1) 0 (\omega+1)) \,) = e^\omega o(\, (\omega +1) 0 (\omega+1)\, ) = e^\omega (e^{\omega+1}1\cdot 2)$;
\item
$o(\omega^2) = o(\omega^2 \uparrow 0) = e^{\omega^2} o(0) = e^{\omega^2} 1$;
\item
$o(\, (\omega\cdot 2 +1) \omega (\omega\cdot 2 +1) 0 \omega^2\, ) = o(\omega^2) + o(\, (\omega\cdot 2 +1) \omega (\omega\cdot 2 +1)\, ) = e^{\omega^2}1 + e^\omega (e^{\omega+1}1 \cdot 2)$.
\end{enumerate}
\end{proof}}{\begin{proof} Theorem \ref{theorem:RecursiveHyperexponential} yields  $o(A) = e^{\omega^2}1 + e^\omega (e^{\omega+1}1 \cdot 2)$. Then, the result follows from Theorem \ref{istheexact} above and Example \ref{example:anExactSequence}.\end{proof}}

Recall that we can recast our example in terms of the more familiar Veblen functions as
\[
\vec o (A) = \la \varphi_2(1) + \varepsilon_{\varepsilon_\omega + \varepsilon_\omega}, \varepsilon_{\varepsilon_\omega + \varepsilon_\omega},  \ldots, \varepsilon_\omega + \varepsilon_\omega,  \varepsilon_\omega, \ldots , \omega, 1, 0 \ldots \ra .
\]

Hyperexponentials give us {\em lower bounds} on $\le$-exact sequences.
The value of ${\formerOmega}_\xi(A)$ fully determines the values of ${\formerOmega}_\zeta(A)$ for $\zeta > \xi$ but not vice versa. However for $\zeta > \xi$ we do have a lower-bound on ${\formerOmega}_\xi(A)$:

\begin{theorem}
Given a worm ${A}$ and ordinals ${\x},\zeta$, ${\formerOmega}_{\x}({A})\geq \ex^{\zeta}{\formerOmega}_{{\x}+\zeta}({A})$.
\end{theorem}

\proof
Towards a contradiction, assume that there is a worm $A$ and ordinals $\xi{<}\zeta$ such that ${\formerOmega}_\x(A){<}\ex^{{-{\x}+\zeta}}{\formerOmega}_\zeta(A)$. Then, by Theorem \ref{logexp}, $\le^{-\x+\zeta}{\formerOmega}_\x(A)<{\formerOmega}_\zeta(A)$.
But this is impossible by Theorem \ref{istheexact}, given that $\le^{-\x+\zeta}{\formerOmega}_\x(A)={\formerOmega}_\zeta(A)$.
\qed



\bibliographystyle{plain}
\bibliography{References}

\end{document}